\documentclass{article}
\usepackage[margin=1in]{geometry}
\usepackage[normalem]{ulem}
\usepackage{lineno}
\usepackage{soul}
\usepackage{graphicx}
\usepackage{xcolor}
\usepackage{amssymb, amsmath, amsthm, hyperref, euscript, color}
\usepackage[dvipsnames]{xcolor}
\usepackage{amscd}
\usepackage{tikz-cd}
\usepackage{bbm}
\usepackage{enumitem}
\usepackage[english]{babel} 
\usepackage{mathtools}
\usepackage{makecell}
\usepackage[none]{hyphenat}
\usepackage{comment}
\usepackage[percent]{overpic}
\usepackage{tikz, tikz-cd}

\usepackage{caption}
\numberwithin{equation}{section}

\newtheoremstyle{theor}{6pt plus 1pt minus 1pt}{6pt plus 1pt minus 1pt}{\slshape}{}{\bfseries}{.}{5pt plus 1pt minus 1pt}{}
\newtheoremstyle{def}{6pt plus 1pt minus 1pt}{6pt plus 1pt minus 1pt}{}{}{\bfseries}{.}{5pt plus 1pt minus 1pt}{}
\newtheoremstyle{rmk}{6pt plus 1pt minus 1pt}{6pt plus 1pt minus 1pt}{}{}{\bfseries}{.}{5pt plus 1pt minus 1pt}{}
\newtheoremstyle{claim}{6pt plus 1pt minus 1pt}{6pt plus 1pt minus 1pt}{\slshape}{}{\bfseries}{.}{5pt plus 1pt minus 1pt}{}

\theoremstyle{theor}
\newtheorem{newstatement}{newstatement}[section]
\newtheorem{lemma}[newstatement]{Lemma}
\newtheorem{theorem}[newstatement]{Theorem}
\newtheorem{corollary}[newstatement]{Corollary}
\newtheorem{proposition}[newstatement]{Proposition}

\theoremstyle{def}
\newtheorem{definition}[newstatement]{Definition}

\theoremstyle{rmk}
\newtheorem{remark}[newstatement]{Remark}
\newtheorem*{remark*}{Remark}
\newtheorem{example}[newstatement]{Example}
\newtheorem*{example*}{Example}

\theoremstyle{claim}

\theoremstyle{theor}
\newtheorem{thm}{Theorem}

\expandafter\let\expandafter\oldproof\csname\string\proof\endcsname
\let\oldendproof\endproof
\renewenvironment{proof}[1][\proofname]{%
  \oldproof[\slshape #1]%
}{\oldendproof}
\def\provedboxcontents#1{$\square$}

\makeatletter
\newsavebox\myboxA
\newsavebox\myboxB
\newlength\mylenA

\newcommand*\xoverline[2][0.75]{%
    \sbox{\myboxA}{$\m@th#2$}%
    \setbox\myboxB\null
    \ht\myboxB=\ht\myboxA%
    \dp\myboxB=\dp\myboxA%
    \wd\myboxB=#1\wd\myboxA
    \sbox\myboxB{$\m@th\overline{\copy\myboxB}$}
    \setlength\mylenA{\the\wd\myboxA}
    \addtolength\mylenA{-\the\wd\myboxB}%
    \ifdim\wd\myboxB<\wd\myboxA%
       \rlap{\hskip 0.5\mylenA\usebox\myboxB}{\usebox\myboxA}%
    \else
        \hskip -0.5\mylenA\rlap{\usebox\myboxA}{\hskip 0.5\mylenA\usebox\myboxB}%
    \fi}
\makeatother

\newcommand{\Z}{\mathbb{Z}}

\DeclareMathOperator{\coker}{coker}

\DeclareMathOperator{\Hom}{Hom}

\DeclareMathOperator{\im}{im}

\let\geq\geqslant
\let\leq\leqslant

\newcommand{\darrow}[1]
{\mathchoice
{\xlongrightarrow{\hbox{\raisebox{0pt}[0pt][0pt]
  {\smash{$\scriptstyle\mkern1mu#1
  \mkern3mu:\mkern2mu1\mkern1.5mu$}}}}}
{\xlongrightarrow{\hbox{\raisebox{-1.5pt}[0pt][0pt]
  {\smash{$\scriptstyle\mkern2mu#1
  \mkern3mu:\mkern2mu1\mkern2.5mu$}}}}}
{}{}}

\tikzcdset{every label/.append style = {font=\normalsize,inner sep=0.75ex}}
\tikzcdset{every matrix/.append style = {row sep=5ex,column sep=1em}}

\newif\ifrevisions
\revisionstrue
\ifrevisions
  
  \newcommand{\del}[1]{{\color{red}\sout{#1}}}
  \newcommand{\query}[1]{{\color{Magenta}[\textbf{Query:} #1]}}
\else
  
  \newcommand{\del}[1]{}
  \newcommand{\query}[1]{}
\fi

\title{Cobordism groups of dihedral branched covers}
\author{Valentina Bais and Alexandra Kjuchukova}

\date{}

\begin{document}

\maketitle

\begin{abstract}
For every integer $n \geq 1$, we compute the cobordism groups of dihedral $n$-fold branched covers of $S^3$ with oriented and non-oriented branching sets. We show that the groups are cyclic, generated by the $n$-fold connected dihedral covers of  $(2,n)$-torus links. The isomorphism types of the groups are detected using explicit cobordism invariants defined in terms of
the Seifert forms on the branching sets, generalizing Cappell-Shaneson characteristic knots associated to dihedral covers. 
\end{abstract}

\section{Introduction}
In this note, we study the cobordism groups of dihedral branched covers (see Definition~\ref{def:dihedralcover}). The underlying sets of these groups are obtained from the set of all $D_n$-covers of $S^3$  (see Definition~\ref{def:Dncover}) under the natural equivalence relation: two covers $p_0 \colon Y_0\to S^3$ and $p_1\colon Y_1 \to S^3$ are identified if and only if there is a cobordism $W$ between $Y_0$ and $Y_1$ and a $D_n$-cover $q: W\to S^3 \times I$ branched over a properly embedded, smooth, orientable surface and restricting to the given covers on the boundary. We distinguish two cases. When the branch sets of $p_0$, $p_1$ and $q$ are oriented and the orientations are preserved under the cobordism, we have the oriented cobordism group $Cob_{D_n}(S^3)$; when the branch sets are unoriented, we have the non-orientable version $Cob^{\leftrightarrow}_{D_n}(S^3)$. Both sets can then be endowed with an abelian group structure under the operation of taking the split union of the branched sets, with the induced monodromies. See Definitions~\ref{def:orientedcob} and~\ref{def:unorcob} for a formal description of $Cob_{D_n}(S^3)$ and $Cob^{\leftrightarrow}_{D_n}(S^3)$, respectively. 

Cobordism groups of branched covers were also studied in \cite{A}, \cite{B}, \cite{N}, usually by introducing classifying spaces and computing their homotopy groups. In general, the cobordism groups themselves are rarely determined. One exception are irregular dihedral covers in degree $3$, which coincide with $3$-fold simple branched covers.  Moreover, for $n=3$ the cobordism groups of regular and irregular dihedral covers are isomorphic. The oriented cobordism group $Cob_{D_3}(S^3)$ therefore coincides with the oriented cobordism group of simple $3$-fold branched covers of $S^3$. The latter was computed by Hilden-Little in \cite[Theorem~9]{HL} in dimensions $3$, $4$ and $5$. In dimension $3$, we compute $Cob_{D_n}(S^3)$ for all integers $n \geq 1$.

\begin{thm}\label{thm:main}
   For every integer $n \geq 1$, there is an injective group homomorphism
    \[\kappa_n \colon Cob_{D_n}(S^3) \longrightarrow \Z_{2n},\] defined in Equation~\ref{eq:k_n-def}, 
    whose image is the subgroup of even residues if $n$ is odd and all of $\Z_{2n}$ if $n$ is even. The reduction of $\kappa_n$ modulo $n$ descends to a group isomorphism
    \[\kappa^{\leftrightarrow}_n \colon Cob^{\leftrightarrow}_{D_n}(S^3) \xlongrightarrow{\cong} \Z_n.\]
    In particular,
    \[Cob_{D_n}(S^3)\cong \begin{cases} \Z_n & n \text{ odd},\\ \Z_{2n} & n \text{ even},\end{cases} \qquad\quad Cob^{\leftrightarrow}_{D_n}(S^3)\cong \Z_n.\]

    In all cases, a generator is represented by the connected $n$-fold dihedral cover branched over an oriented $(2,n)$-torus link, endowed with the Fox coloring determined by the mod $n$ characteristic knot $\beta_n$; see Example~\ref{ex:generator}. Its invariant $\kappa_n$ is $1-n$.
\end{thm}

\begin{remark}\label{rmk:cob}
    In Corollary \ref{cor:cob}, we compute one more cobordism group: the one where the underlying set is the set of $D_n$-covers branched over \textit{unoriented} links and the equivalence relation identifies $p_0$ and $p_1$ if and only if there is a $D_n$-cover cobordism $q \colon W \rightarrow S^3 \times I$ branched over an orientable but {\it unoriented} surface, see Definition \ref{def:unorientcob}. We show that this group is isomorphic to $Cob^{\leftrightarrow}_{D_n}(S^3)$.
\end{remark}

The homomorphism $\kappa_n$ is explicit and cobordism classes are easily computable. In Theorem \ref{thm:recovery-of-dihedral-homs}, we generalize a result by Cappell and Shaneson \cite{CS} and show that every dihedral cover $p \colon Y \rightarrow S^3$ branched along a link $L\subset S^3$ is determined by a \textit{mod $n$ characteristic link} $\beta$ for $L$ (see Definition~\ref{def:charlink}). The link $\beta$ is contained in a Seifert surface $S$ for $L$. The invariant $\kappa_n$ is given in terms of the self-linking number of $\beta$ with respect to the push-off determined by $S$: if $A$ denotes a symmetrized Seifert matrix of $S$, then 
\begin{equation}\label{eq:k_n-def}
    \kappa_n([p])=\frac{1}{n}[\beta]^t \,A\,[\beta] \pmod{2n}.
\end{equation}

Notice that $Cob_{D_n}(S^3)$ and $Cob^{\leftrightarrow}_{D_n}(S^3)$ are isomorphic for every odd integer $n$. On the other hand, when $n$ is even, forgetting the orientation on the branching set induces a map
\[Cob_{D_n}(S^3) \longrightarrow Cob^{\leftrightarrow}_{D_n}(S^3)\]
corresponding to the mod $n$ reduction
\[\Z_{2n} \longrightarrow \Z_n.\] 
The kernel of this map is generated by a dihedral cover branched over two copies of a Hopf link as the ones given in Example \ref{ex:Hopf}.

As an application of Theorem \ref{thm:main}, we show that every $3$-manifold arising as the total space of a $D_n$-cover over $S^3$ admits a $D_n$-cover over $S^3$ which extends over $D^4$, i.e. which is trivial in the (oriented) cobordism group. In order to do that, we introduce a set of moves that modify a given dihedral cover without altering its total space, generalizing Montesinos moves for $3$-fold simple branched covers, see \cite{M} and \cite{P}.
\begin{thm}\label{thm:C}
    Let $p \colon Y \rightarrow S^3$ be a branched $D_n$-cover, where $Y$ is a closed connected $3$-manifold. Then there exists a branched $D_n$-cover $p' \colon Y \rightarrow S^3$ which extends to a $D_n$-cover of $D^4$ branched over a smooth, properly embedded oriented surface. 
\end{thm}

Our methods also easily recover the following well-known result, see \cite{BP}, \cite{M2}.
\begin{corollary}\label{cor:4-ball-covered}
    Every closed oriented 3-manifold $Y$ bounds a 4-manifold which is an irregular 3-fold cover of $B^4$. 
\end{corollary}

\subsection{Paper organization}
In Section~\ref{sec:preliminaries}, we provide background and basic definitions. Section~\ref{sec:char} recalls the definition of Cappell-Shaneson characteristic knots and introduces the generalization we need. In Section~\ref{sec:extensions} we discuss covers which represent the trivial class in the cobordism group, recalling a result of \cite{KO}. The invariant $\kappa_n$ is defined in Section~\ref{sec:meat} and the proof of the main theorem is contained in Section~\ref{sec:proof}. Theorem \ref{thm:C} and Corollary \ref{cor:4-ball-covered} are proved in Section \ref{sec:mont}.

\section{Preliminaries}\label{sec:preliminaries}
We recall some standard definitions and classical results which we will rely on later. We also give a formal definition of the oriented and non-orientable cobordism groups.  
\subsection{Branched covers}

The following is the standard definition of a branched cover in the smooth category. 
\begin{definition} Let $M$ and $N$ be smooth compact $n$-manifolds and $d \geq 1$ an integer. A smooth map
\[p \colon M \longrightarrow N\] is called a $d$-fold \textit{branched cover} if there exists a properly embedded smooth $(n-2)$-submanifold $B_p \subset N$ with the following properties: 
\begin{itemize}
    \item[(a)] the restriction of $p$ to $p^{-1}(N \setminus B_p)$ is an ordinary (unbranched) $d$-fold covering map;
    \item[(b)] in a neighbourhood of every $x \in p^{-1}(B_p)$, $p$ is smoothly equivalent to the map
    \[ \mathbb{R}^{n-2} \times \mathbb{C} \longrightarrow \mathbb{R}^{n-2} \times \mathbb{C}\]
    \[(t,z)\mapsto (t, z^k)\]
    for some positive integer $k \leq d$, called the \textit{local branching index} of $p$ at $x$.
\end{itemize} 

The integer $d$ is called the \textit{degree} of the cover, while the submanifold $B_p\subset N$ is the \textit{branch set} of $p$. The \textit{monodromy} of $p$ is the representation
\[\omega_p \colon \pi_1(N \setminus B_p) \longrightarrow S_d\]
associated to the ordinary cover $p|_{p^{-1}(N \setminus B_p)}$, where $S_d$ denotes the permutation group of $d$ elements. One may also replace $S_d$ by a subgroup containing the image of $\omega_p$ and still refer to the map as the monodromy of $p$.
\end{definition}
Sometimes, it will be convenient to work with \textit{equivalence classes} of branched covers. We say that two branched covers $p_0$ and $p_1$ are \textit{equivalent} if and only if there exists a commutative diagram

\begin{center}
\begin{tikzcd}
M_0 \arrow[d, "p_0"'] \arrow[rr, "\widetilde \varphi"] &  & M_1 \arrow[d, "p_1"] \\
N_0 \arrow[rr, "\varphi"]                              &  & N_1                 
\end{tikzcd}
\end{center}
where $\widetilde \varphi$ and $\varphi$ are diffeomorphisms. 

In this paper we focus on dihedral branched covers. Throughout, we denote the dihedral group of order $2n$ by $D_n$ and we fix the following notation for its presentation: 
\begin{equation}\label{eq:Dn}
D_n = \langle x,y \mid x^2=1,\ y^n=1,\ xyx = y^{-1} \rangle.
\end{equation}
We will refer to $x$ and $y$ as the preferred reflection and the generating rotation, respectively. We also fix, once and for all, the subgroup \[H:=\langle x \rangle \leq D_n\] generated by the preferred reflection.

\begin{definition}[Irregular dihedral branched cover]\label{def:dihedralcover}
    Let $M,N$ be two compact oriented smooth $m$-manifolds. A branched cover $p \colon M \rightarrow N$ with branch set a smoothly embedded codimension-$2$ submanifold $B_p \subset N$ is called an \textit{$n$-fold irregular dihedral branched cover} if its monodromy representation 
    \[
    \omega_p \colon \pi_1(N \setminus B_p)\longrightarrow D_n
    \]  
    sends every meridian of $B_p$ to a reflection in $D_n$, and the restriction of $p$ over $N \setminus B_p$ is the (possibly disconnected) $n$-fold cover with fiber $D_n/H$ associated, under the classification of ordinary covers, to the action of $\pi_1(N\setminus B_p)$ on $D_n/H$ through $\omega_p$. When $\omega_p$ is surjective this cover is connected and corresponds to the subgroup $\omega_p^{-1}(H)$.
    \end{definition}

For our purposes, it is not possible to require $\omega_p$ to be surjective. Non-surjectivity is needed for defining the addition operation in the cobordism groups and for the existence of an identity element. 

The invariants computed in this paper are sensitive to the monodromy homomorphism $\omega_p$ and are not determined by the covering map $p$ alone. Therefore, we adopt the following convention.

\begin{definition}[$D_n$-cover]\label{def:Dncover}
A \textit{$D_n$-cover} of $N$ with branch set $B \subset N$ is an inner-automorphism class of homomorphisms
\[\omega \colon \pi_1(N \setminus B)\longrightarrow D_n\]
sending every meridian of $B$ to a reflection. Equivalently, a $D_n$-cover is a regular dihedral branched cover $\widetilde p \colon \widetilde M \rightarrow N$ together with a covering action of $D_n$ on $\widetilde M$, considered up to $D_n$-equivariant isomorphism.  When $\omega$ is surjective, it is the connected cover determined by the kernel of $\omega$. The associated irregular cover of a $D_n$-cover, with respect to a fixed conjugacy class of reflection subgroup $H\cong\Z_2$, is the $n$-fold cover of Definition~\ref{def:dihedralcover}, that is, the branched completion of the cover with fiber $D_n/H$ induced by $\omega$.
\end{definition}

\begin{remark}\label{rmk:covesvsmonodromy}
If one remembers only the map $p$, its monodromy is identified with $D_n$ only up to $Aut(D_n)$. The automorphism $x \mapsto x$, $y \mapsto y^k, k\in (\Z_n)^x$, which is outer for $k\neq\pm 1$, preserves $H$ and hence the map $p$; but it replaces the characteristic link $\beta$ of Theorem \ref{thm:recovery-of-dihedral-homs} by $k\beta$ and hence $\kappa_n$ by $k^2\kappa_n$. Already for the torus link $L = T(2,5)$ and $n = 5$ the colorings determined by $\beta_5$ (of Example \ref{ex:generator}) and $2\beta_5$ induce the same covering map, but the corresponding $\kappa_5$ invariants take the distinct values 6 and 4 in $\Z_{10}$. Thus, $\kappa_n$ is an invariant of the pair $(p,\omega)$, not of $p$ alone. 
\end{remark}

When $n$ is odd, all reflections in $D_n$ are conjugate and fixing $H$ is merely a notational convenience. For $n$ even, there are two irregular dihedral covers for each $D_n$-cover. They are associated to the two conjugacy classes of reflections in $D_n$. Fixing $H$ globally, instead of recording a conjugacy class of reflections as part of the data for each cover, makes the addition operation in the cobordism group well-defined (see Definition~\ref{def:orientedcob}).

We give a geometric description of the covers we study. A $D_n$-cover has local branching index $2$ at each of the $n$ points in the pre-image of a point $b$ on the branch set $B_p$. The action of the rotation subgroup in $D_n$ permutes these $n$ points.

In order to describe the local degrees at $p^{-1}(b)$ in an irregular dihedral cover, we have to consider several cases. When $n$ is odd, all irregular dihedral covers associated to a fixed monodromy $\omega \colon \pi_1(S^3\backslash L)\to D_n$ are equivalent. When $\omega$ is surjective, for every $b\in B_p,$ the pre-image $p^{-1}(b)$ in the irregular dihedral cover consists of one point of branching index $1$ and $\frac{n-1}{2}$ points of branching index $2$.

When $n$ is even, as we already discussed, a surjective monodromy $\omega\colon \pi_1(S^3\backslash L)\to D_n$ induces two equivalence classes of irregular dihedral covers, corresponding to the two conjugacy classes of $\Z_2$ subgroups in $D_n$. Each equivalence class of irregular cover is a $\Z_2$ quotient of the regular dihedral cover induced by $\omega$ under the action of the relevant $\Z_2$ subgroup. If this action restricted to $\widetilde p^{-1}(b)$ has no fixed points, then $p^{-1}(b)$ consists of $\frac{n}{2}$ points of index 2. Alternatively, the $\Z_2$ on $\widetilde p^{-1}(b)$ has 2 fixed points, in which case  $p^{-1}(b)$ consists of 2 points of branching index 1 and $\frac{n-2}{2}$ points of branching index 2. We note that both cases can occur (for different choices of $b$) in the same irregular dihedral cover since meridians of distinct components of $L$ may be mapped to non-conjugate reflections.

Lastly, for a non-surjective dihedral representation $\omega$ of a link group, since we require that all meridians are sent to reflections, the image is isomorphic to $D_m$ for some divisor $m$ of $n$ with $1 \leq m < n$. The case $D_1\cong\mathbb Z_2$ corresponds to the trivial dihedral homomorphism (sending each meridian to the same reflection), which has image $\Z_2$. Thus, when $\omega$ is not surjective, one obtains disconnected (regular or irregular) dihedral covers whose connected components correspond to the orbits of $\im \omega$ acting on $D_n$ (resp. on $D_n/H$). The regular cover consists of $n/m$ copies of the regular $D_m$-cover. The following example illustrates what the connected components of the irregular dihedral cover may look like.

\begin{example}
    Consider the case where $\omega\colon \pi_1(S^3\backslash L)\to D_8$ has image isomorphic to $D_4$. In this case, the regular $16$-fold $D_8$ cover induced by $\omega$ consists of two copies of the regular $D_4$ cover. An appropriate choice of $\Z_2\subset D_8$ acts by identifying the two components. In this case, the irregular $D_8$ cover is diffeomorphic to the regular $D_4$ cover. The second equivalence class of a $\Z_2\subset D_8$ subgroup results in a disconnected irregular $D_8$ cover which is diffeomorphic to the disjoint union of two irregular $D_4$ covers, one from each equivalence class. 
\end{example}

Our main objects of study are introduced next.

\begin{definition}[Oriented dihedral cobordism group]\label{def:orientedcob}
Let $n \geq 1$ be an integer. The \textit{oriented cobordism group of dihedral $n$-fold branched covers of $S^3$} is
\[Cob_{D_n}(S^3)
=
\left\{
p \colon Y\to S^3
\;\middle|\;
\substack{
p \text{ is a $D_n$-cover of $S^3$}\\
\text{branched over an oriented link}
}
\right\}\big/\!\sim.
\]

where $p_0\sim p_1$ if and only if there exists a $D_n$-cover \[q \colon W \darrow{n} S^3 \times I\] branched over a smoothly embedded oriented surface $C$, such that the restriction of $q$ over $S^3\times \{0\}$ (resp. $S^3 \times \{1\}$) equals $p_0$ (resp. $p_1$), and such that $\partial C = -B_{p_0} \sqcup B_{p_1}$ as oriented links.  

\begin{figure}
    \centering
    \includegraphics[width=0.55\linewidth]{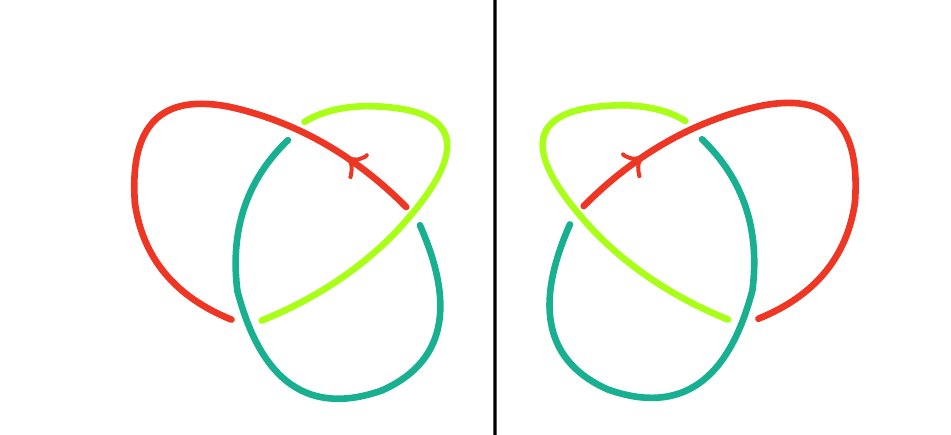}
    \caption{A Fox colored trefoil knot and its Fox colored mirror. By Proposition~\ref{prop:sumdef}, the two $D_n$-covers they determine represent inverse classes in $Cob_{D_n}(S^3)$.}
    \label{fig:mirror}
\end{figure}

We endow $Cob_{D_n}(S^3)$ with an abelian group structure as follows. Given $[p],[q]\in Cob_{D_n}(S^3)$, we set
\[[p]+[q]=[p\sqcup q]\]
where $p\sqcup q$ denotes the $n$-fold cover of $S^3$ branched over the split union $B_p \sqcup B_q$ of the Fox colored branch sets of $p$ and $q$. This induces a well-defined abelian group structure on $Cob_{D_n}(S^3)$, see Proposition \ref{prop:sumdef}.
\end{definition}

We emphasize the fact that in Definition \ref{def:orientedcob}, we require that the restriction of the orientation of the branch set $B_q$ to the boundary agrees with the reverse of the orientation of $B_{p_0}$ and with the orientation of $B_{p_1}$, respectively; this is the usual convention for an oriented link cobordism, and is forced by the fact that an oriented surface in $S^3 \times I$ induces opposite boundary orientations at the two ends. This requirement is waived in the non-oriented version of the dihedral cobordism group, defined next.

\begin{definition}[non-orientable dihedral cobordism group]\label{def:unorcob}
    Let $n \geq 1$ be an integer. The \textit{non-orientable cobordism group of dihedral $n$-fold branched covers of $S^3$} is
\[Cob^{\leftrightarrow}_{D_n}(S^3)
=
\left\{
p \colon Y\to S^3
\;\middle|\;
\substack{
p \text{ is a $D_n$-cover of $S^3$}\\
\text{branched over an unoriented link}
}
\right\}\big/\!\sim.
\]

where $p_0\sim p_1$ if and only if there exists a $D_n$-cover \[q \colon W \darrow{n} S^3 \times I\] branched over a smoothly embedded surface, such that the restriction of $q$ over $S^3\times \{0\}$ (resp. $S^3 \times \{1\}$) equals $p_0$ (resp. $p_1$). 

We endow $Cob^{\leftrightarrow}_{D_n}(S^3)$ with an abelian group structure, in a manner identical to that in Definition \ref{def:orientedcob}, see Proposition \ref{prop:sumdef}.  
\end{definition}

In the following Proposition, we show that the abelian group structures on $Cob_{D_n}(S^3)$ and $Cob^{\leftrightarrow}_{D_n}(S^3)$ are well-defined.

\begin{proposition}\label{prop:sumdef}
The split union of Fox colored branch sets descends to a well-defined abelian group operation on $Cob_{D_n}(S^3)$ and on $Cob^{\leftrightarrow}_{D_n}(S^3)$.
\end{proposition}
\begin{proof}
We show that the operation is well defined. With this purpose in mind, we prove that, given a link $L\subset S^3$ and two representations \[\rho, \rho' \colon \pi_1(E_L) \longrightarrow D_n\]
that differ by post composition with an inner automorphism of $D_n$, there is a cobordism between the $D_n$-covers $(L, \rho)$ and $(L, \rho')$. As a consequence, we get that, if $(L_1, \rho_1)=(L_1, \rho_1')$ as $D_n$-covers, then the split union of the Fox colored links $(L_1, \rho_1)$ and $(L_2, \rho_2)$ is cobordant to the split union of  $(L_1, \rho_1')$ and $(L_2, \rho_2)$, where a cobordism can be obtained by putting together a cobordism between $(L_1, \rho_1)$ and $(L_1, \rho_1')$ with the identity cobordism of $(L_2, \rho_2)$ (and analogously for the second summand). 

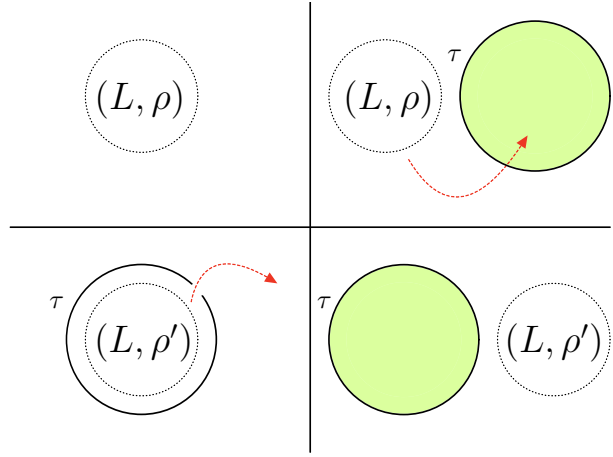
\begin{figure}
    \centering
     \begin{overpic}[width=0.55\linewidth]{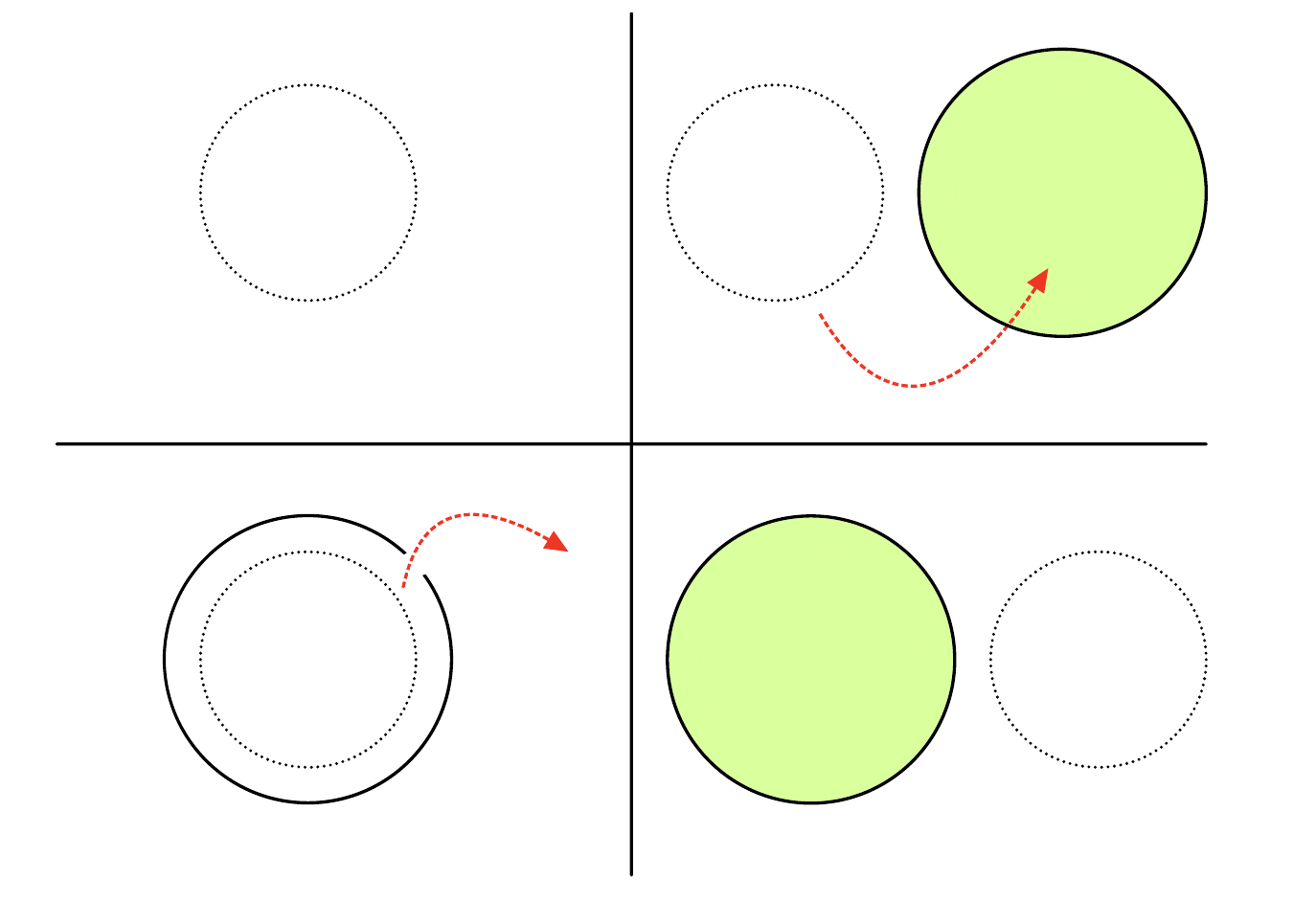}
        \put(52.8,52){\Large $(L, \rho)$}
        \put(16.6,52){\Large $(L, \rho)$}
        \put(16.5,16.5){\Large $(L, \rho')$}
        \put(76.3,16.5){\Large $(L, \rho')$}
        \put(10,23){$\tau$}
        \put(49,23){$\tau$}
        \put(68,59){$\tau$}
    \end{overpic}
    
    \caption{A cobordism between $(L, \rho)$ and $(L, \rho')$. Top left: the colored link $L$ before the $0$-handle appears. Top right: $L\cup U \subset S^3\times\{1\}$. We slide $L$ under $U$; the result is seen bottom left. We have conjugated the coloring of $L$ by $\tau$. Bottom right: the result of sliding $L$ over $U$.}
    \label{fig:cobordism}
\end{figure}

Since conjugation by a rotation in $D_n$ can be always expressed as the composition of two conjugation by reflections, it is enough to prove the claim for the case when $\rho=c \circ \rho'$ for an inner automorphism $c$ induced by conjugating by a reflection $\tau \in D_n$. In this case, a cobordism from $(L, \rho)$ to $(L, \rho')$ can be constructed as follows. Start with the Fox colored link $(L, \rho)$ and consider an elementary cobordism given by a $2$-dimensional $0$-handle $h^0_2$ embedded in a ball in $S^3 \times [0,1]$ split from $L$. We color all meridians of $h^0_2$ by $\tau$. We may assume that $\partial h^0_2$ is embedded in $S^3\times \{1\}$, producing a split union between $(L, \rho)$ and an unknot $U$ labeled by $\tau$, see Figure~\ref{fig:cobordism}. Now, perform an ambient isotopy that slides a $3$-ball containing $L$ first under, then over $U$ as depicted in Figure~\ref{fig:cobordism}.  If we describe $\rho$ by means of its value on a Wirtinger system of generators for $\pi_1(E_{L \sqcup U})$, we have that the Fox coloring of $L$ is changed by conjugation with $\tau$. At this point, we cap off $U$ via a colored $2$-handle cobordism, where the interior of the $2$-handle involved, $h^2_2$, is contained in $S^3\times(1, 2]$ and $h^2_2$ is isotopic rel. boundary to the green shaded $2$-disk at the bottom right of Figure \ref{fig:cobordism} via an isotopy supported in a 4-ball disjoint from $L$. 

The operation of split union is hence well defined on equivalence classes of $D_n$-covers. Commutativity and associativity clearly hold. The trivial $D_n$-cover (induced by a monodromy with image $\Z_2$) extends over any orientable surface in $D^4$ bounding a link $L\subset S^3$, so its class is the identity element. Finally, given any link $L$ equipped with a $D_n$-cover, this cover extends over $L\times [0,1]\subset S^3\times [0,1]$, so the inverse of $L$ is the mirror of $L$ with the orientation and Fox coloring induced by the mirroring operation; see Figure~\ref{fig:mirror} for an example.
\end{proof}

\subsection{Spanning surfaces and linking pairings}
Let $L\subset S^3$ be an oriented link and $S$ an oriented Seifert surface for $L$.
\begin{definition}[Seifert pairing]\label{def:Seifertpairing} The \textit{Seifert pairing} of $S$ is the bilinear form
\[\mathcal{V} \colon H_1(S;\Z) \times H_1(S;\Z) \longrightarrow \Z\]
    \[([a],[b]) \mapsto lk(a, b^+),\]
    where $a, b \subset S$ are simple closed curves (not necessarily connected) representing $[a]$ and $[b]$, respectively, and $b^+\subset S^3 \setminus S$ denotes a push-off of $b$ along the positive normal direction of $S$. We denote by $V$ a matrix representing the Seifert pairing $\mathcal{V}$ with respect to a fixed basis of $H_1(S;\Z)$.
    
    The matrix
    \[A=V +V^t\]
    is called a \textit{symmetrized Seifert matrix for $S$}.
\end{definition}
The following example will be of much use in this paper.

\begin{example}[$(2,n)$- torus links]\label{ex:torustake1}

For every integer $n \geq 1$, we denote by $T(2,n)$ the $(2,n)$-torus link. Let $S(2,n)$ be the standard oriented Seifert surface for $T(2,n)$ and let $\alpha_1, \dots, \alpha_{n-1}\subset S(2,n)$ be the basis of $H_1(S;\Z)$ induced by the simple loops $a_1, \dots, a_{n-1}\subset S(2,n)$ depicted in Figure \ref{fig:torus}. One easily checks that the associated Seifert matrix is the $(n-1)\times (n-1)$ matrix 
    \begin{equation*} V_n=
        \begin{pmatrix}
-1 & 1 & 0 & 0 & \dots & 0 \\
0 & -1 & 1 & 0 & \dots & 0 \\
\vdots & \ddots & \ddots & \ddots & \ddots & \vdots \\
0 & \dots & \dots & \dots  & 0 & -1
\end{pmatrix}.
    \end{equation*}
    
\end{example}

\begin{figure}
    \centering
     \begin{overpic}[width=0.4\linewidth]{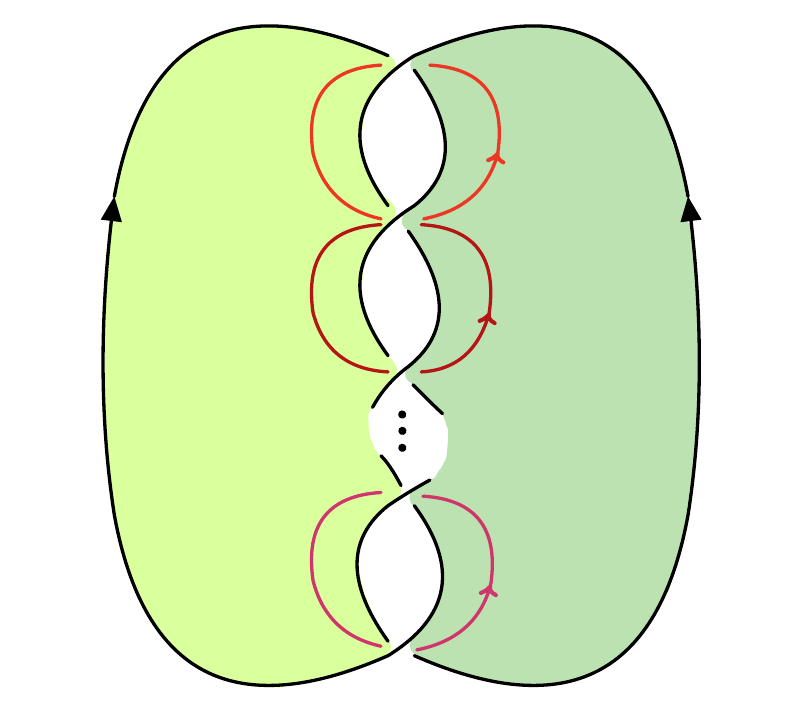}
        \put(65,70){\textcolor{red}{\large $a_1$}}
        \put(65,50){\textcolor{red}{\large $a_2$}}
        \put(65,15){\textcolor{red}{\large $a_{n-1}$}}
    \end{overpic}

    \caption{The torus link $T(2,n)$ with Seifert surface $S(2,n)$. With respect to the displayed basis for $H_1(S(2,n); \Z)$, the characteristic knot $\beta_n$ represents the class $a_1+2a_2+...+ (n-1)a_{n-1}$.}
    \label{fig:torus}
\end{figure}

When $S$ is a possibly non-orientable spanning surface for $L \subset S^3$, we can still define a pairing on $H_1(S;\Z)$ per \cite[Section 2]{GL}, as we now recall. Let $L\subset S^3$ be a link and $S \subset S^3$ a connected spanning surface for $L$. Denote by $\widetilde S \subset S^3 \setminus S$ the total space of the $S^0$-bundle over $S$ induced by the $I$-bundle that is the tubular neighbourhood of $S$ in $S^3$. Let $\tau \colon \widetilde S \rightarrow S$ be the associated bundle map.
\begin{definition}[Gordon--Litherland pairing]\label{def:gordonlitherland}
 The \textit{Gordon--Litherland pairing} on $S$ is the bilinear form
 \[\mathcal{U} \colon H_1(S;\Z) \times H_1(S;\Z) \longrightarrow \Z \]
 \[([a],[b]) \mapsto lk(a, \tau^{-1}(b)),\]
 where $a, b \subset S$ are simple closed curves representing $[a]$ and $[b]$ respectively. After fixing a basis of $H_1(S;\Z)$, we shall denote by $U$ a matrix representing the Gordon--Litherland pairing $\mathcal{U}$ with respect to this basis.
\end{definition}

\begin{remark}\label{rmk:glandseifert}
    It is straightforward to check that $\mathcal{U}$ is symmetric and, when $S$ is orientable,
    \[[a]^t \, A \,  [b]=[a]^t \, U \,  [b]\]
    for every $[a],[b]\in H_1(S;\Z)$, where, $A=V+V^t$ denotes the symmetrized Seifert matrix on $S$. More generally, for $S$ non-orientable and $a$ an orientation-preserving loop in $S$, the pre-image $\tau^{-1}(a)$ consists of the union of the positive and negative push-offs of $a$ with respect to the (trivial) $I$-bundle on a tubular neighborhood of $a$ in $S$. In other words, for $a, b$ two orientation-preserving loops in $S$, the Gordon-Litherland pairing agrees with the Seifert pairing on an orientable subsurface of $S$ containing $a$ and $b$.
\end{remark}

We conclude this subsection by stating a fact about spanning surfaces which will be useful later. Recall that, given an embedded surface $S \subset S^3$, a \textit{tubing} on $S$ is the result of an embedded surgery on $S$. That is, we replace two disjoint closed disks in the interior of $S$, i.e. an embedded $D^2\times S^0$ in $S^\circ$, by a tube $S^1\times D^1$ whose interior is embedded in $S^3\backslash S$. When the surgered $D^2 \times S^0$ is contained in a single connected component of $S$, we will call this operation a \textit{self-tubing} or a \textit{stabilization} of $S$. We also remark that the added 1-handle may or may not be orientable.

\begin{lemma}\label{lem:stab}
    Let $S$ and $S'$ be orientable Seifert surfaces for a given link $L \subset S^3$. Then $S$ and $S'$ become ambiently isotopic after a finite number of stabilizations. Moreover, if $S$ and $S'$ admit orientations which are compatible on $L$, then they become ambiently isotopic after a finite number of oriented stabilizations. 
\end{lemma}

This is well known and we refer the reader to \cite{F} for references to its many proofs; an elementary proof is given in \cite[Theorem~1]{BFK}.

The analogous statement for possibly non-orientable spanning surfaces requires one more local move: taking the boundary connected sum with an unknotted half-twisted band, where the twist can be positive or negative. The isotopy class of the boundary link is unchanged under this operation. We call this operation a \textit{half-twisted band attachment}. 

\begin{lemma}\label{lem:spanstab}
Any two spanning surfaces for a link $L \subset S^3$ are related by a finite sequence of ambient isotopies, additions or removals of tubes, and additions or removals of half-twisted bands.
\end{lemma}

This is due to Gordon and Litherland \cite[Theorem~11]{GL}; an elementary proof is given in \cite{Y}. In contrast with Lemma~\ref{lem:stab}, the intermediate surfaces may be non-orientable or disconnected, and removals as well as additions are allowed.

\subsection{HNN presentations for link groups}
The following lemma will be needed for the proof of the main theorem in Section \ref{sec:char}. This is a standard fact, but we include a proof because we did not locate a reference.

\begin{lemma}\label{lem:hnn}
    Let $L \subset S^3$ be an oriented link and $S \subset S^3$ a connected oriented Seifert surface for $L$. Denote by $E_L$ and $E_S$ the exteriors of $L$ and $S$ in $S^3$, respectively. Then $\pi_1(E_L)$ admits the HNN presentation
    \begin{equation}\label{eq:hnn}
    \pi_1(E_L)\cong \pi_1(E_S)* \Z \langle m \rangle  / \{mi_+(w) m^{-1}=i_-w \text{ for all } w \in \pi_1(S) \},
    \end{equation}
    where $m$, the generator of the infinite cyclic group, is a meridian of $L$. Here, $i_+$ and $i_-$ denote the positive and negative push-offs of $w\in \pi_1(S)$. In this presentation, the basepoint $x_0$ is contained in the interior of $S^+$, the positive pushoff of $S$ in $\partial(S\times I)$, and $w, w^-$ are connected to $x_0$ along $x_0\times I\subset S\times I$.
\end{lemma}
\begin{proof}
    Let $\gamma\subset E_L$ be a simple loop intersecting $S$ transversely at one point in its interior and containing the line segment $x_0\times I$. Consider the decomposition
    \[E_L \cong E_S\cup \nu(S\cup \gamma),\]
    where $\nu(\cdot)$ denotes a closed tubular neighborhood. Since $E_S\cap \nu(S \cup \gamma)$ is connected, we can apply Seifert van--Kampen's Theorem and conclude that
\begin{equation}\label{eq:pres}
      \pi_1(E_L)\cong \pi_1(E_S)*\pi_1(\nu(S\cup \gamma))/\sim,
\end{equation}
    where $\sim$ is the inclusion-induced equivalence relation we will now describe. Denote by $S^+$ and $S^-$ the positive and negative push-offs of $S$, both contained in the boundary of $\nu(S)$. By assumption, the basepoint $x_0$ for all groups involved satisfies $x_0=S^+ \cap \gamma$. Let $m\in \pi_1(\nu(S)\cup \gamma))$ be the class of $\gamma$ endowed with the orientation going from $S^+$ to $S^-$ in the complement of $S$. Then there is an isomorphism
    \[\pi_1(\nu(S)\cup \gamma)\cong \pi_1(S) * \Z \langle m \rangle.\]
    On the other hand, we have that
    \[\pi_1(E_S\cap \nu(S\cup \gamma))\cong\pi_1(S^+ \cup (\gamma \setminus \nu(S))\cup S^-)\cong \pi_1(S^+)*\pi_1(S^-)\]
    and the inclusion-induced homomorphism  
    \[\pi_1(S^+)*\pi_1(S^-)\longrightarrow \pi_1(S^+)*\Z\langle m \rangle \cong \pi_1(S)*\Z\langle m \rangle\]
    restricts to the identity on the first free factor in the domain and sends $\gamma^-$ to $m\gamma^+m^{-1}$ on the second factor, where $\gamma$ is any loop in $S$ and $\gamma^\pm$ are its two push-offs. 
    The conclusion follows.
\end{proof}
For technical purposes, we will also need a higher dimensional analogue of Lemma \ref{lem:hnn}. We first recall the notion of Seifert solid for a properly embedded surface in $D^4$.

\begin{definition}\label{def:ssolid}
  Let $F\subset D^4$ be a properly embedded smooth oriented surface with non-empty boundary. A \textit{Seifert solid} for $F$ is a compact, connected, oriented $3$-manifold $H \subset D^4$ such that $\partial H=F \cup -S$, where $S\subset \partial D^4$ is a compatibly oriented Seifert surface for $F \cap \partial D^4$.
    
\end{definition}

\begin{lemma}\label{lem:hnnh}
  Let $F \subset D^4$ be a properly embedded oriented surface and $H \subset D^4$ a connected oriented Seifert solid for $F$. Denote by $E_F$ and $E_H$ the exteriors of $F$ and $H$ respectively. Then $\pi_1(E_F)$ admits the following HNN presentation
   \[\pi_1(E_F)\cong \pi_1(E_H)* \Z \langle m \rangle  / \{mi_+(w) m^{-1}=i_-w \text{ for all } w \in \pi_1(H)\},\]
     where $m$ is a generator of the infinite cyclic group and $i_+$ and $i_-$ respectively denote the positive and negative push-offs of $w\in \pi_1(H)$. 
\end{lemma}
\begin{proof}
    The proof is analogous to the one of Lemma \ref{lem:hnn} and is hence omitted.
\end{proof}

\section{Characteristic knots and dihedral branched covers}\label{sec:char}

In what follows, let $K \subset S^3$ be an oriented knot, $S \subset S^3$ is a connected oriented Seifert surface for $K$, and $E_K$ and $E_S$ the exteriors of $K$ and $S$, respectively. Choose a basis for $H_1(S;\mathbb Z)$, and let $V$ be the corresponding Seifert matrix of $S$. Let $A = V+V^t$. Whenever $A$ is applied to a class in $H_1(S;\Z)$, the resulting class in  $H_1(E_S;\Z)$ will be written with respect to the  Alexander dual basis to the chosen basis of $H_1(S;\Z)$. With respect to these bases, the map $i_+ + i_- \colon H_1(S;\Z)\rightarrow H_1(E_S;\Z)$ is represented by the matrix $A=V+V^t$.

Cappell and Shaneson introduced \cite{CS} the following tool for studying dihedral quotients of knot groups.

\begin{definition}\label{def:charknot}
Let $n \geq 1$ be an odd integer. A \textit{Cappell--Shaneson mod $n$ characteristic knot} for $(K,S)$ is an oriented simple closed curve $\beta \subset S^{\circ}$ such that the class $[\beta] \in H_1(S;\mathbb Z)$ is primitive and
\[
A\, [\beta] \equiv 0 \pmod n.
\]
\end{definition}

A mod~$n$ characteristic knot $\beta$ for $K$ determines a $D_n$ quotient of the group of $K$ as in Equation~\ref{eq:rho}, first given \cite{CS}. In what follows, we generalize this idea to the setting where $K$ and $\beta$ are links and $n$ can have either parity. This generalized notion of characteristic links allows us to describe all dihedral representations of link groups. 

\begin{definition}\label{def:charlink}
Let $n \geq 1$ be an integer. A \textit{mod $n$ characteristic link} for $(L,S)$ is an oriented $r$-component link
\[
\beta := \beta_1 \cup \cdots \cup \beta_r \subset S^{\circ}
\]
whose homology class
\[
[\beta] = [\beta_1] + \cdots + [\beta_r] \in H_1(S;\mathbb Z)
\]
satisfies
\[
A\, [\beta] \equiv 0 \pmod n.
\]
Two mod $n$ characteristic links $\beta, \beta' \subset S^\circ$ are \textit{equivalent} if and only if \[[\beta]\equiv[\beta'] \pmod{n}.\]
\end{definition}

We now explain how to construct a dihedral cover branched along $L$ using a mod $n$ characteristic link $\beta$ for $(L,S)$ and a preferred reflection $\tau \in D_n$. Recall that Lemma \ref{lem:hnn} gives us a HNN presentation
\[\pi_1(E_L)\cong \pi_1(E_S)* \Z \langle m \rangle  / \{mi_+(w) m^{-1}=i_-w \text{ for all } w \in \pi_1(S) \}.\]
Given a mod $n$ characteristic link $\beta$ and a preferred reflection $\tau \in D_n$, one can define a representation
\[\rho_{\beta,\tau} \colon \pi_1(E_L) \longrightarrow D_n\]
 by the following formula:

\begin{equation}\label{eq:rho}
\rho_{\beta,\tau}(m)=\tau, \qquad \rho_{\beta,\tau}(w)=y^{lk(\beta,u)} \quad \text{ for every } u \in \pi_1(E_S),
\end{equation}
where $lk(\beta,u)$ denotes the linking number in $S^3$ of $\beta$ with a loop representing $u$. In the case where $K$ is a knot and the dihedral homomorphism is surjective, it follows that $n$ is odd and the above definition coincides with the one given in \cite{CS}. Indeed, for $n$ odd and $\tau$ a generating reflection, all choices of $\tau$ define equivalent dihedral quotients, so this piece of information can be omitted from the notation. 

 We remark that $\rho_{\beta,\tau}$, which we define on $\pi_1(E_S)* \Z$, descends to a map on $\pi_1(E_L)$. Indeed, we have that

\[\rho_{\beta,\tau}(m w^+ m^{-1}) = \tau y^{lk(\beta, w^+)} \tau^{-1} = y^{-lk(\beta, w^+)} \quad \text{and} \quad
\rho_{\beta,\tau}(w^-) = y^{lk(\beta, w^-)}.\]
Pairing $[\beta]$ with any $[w] \in H_1(S;\Z)$ via the symmetrized Seifert matrix $A$ gives
\[
[\beta]^T  A  [w] = lk(\beta^+, w) + lk(\beta^-, w) = lk(\beta, w^+) + lk(\beta, w^-) \equiv 0 \pmod n
\]
 for all $w\in \pi_1(S)$ by the definition of $\beta$. This means that
\[
lk(\beta, w^+) + lk(\beta, w^-) \equiv 0 \pmod n \quad \text{for all } w \in \pi_1(S)
\]
and the HNN relation $\rho_{\beta,\tau}(m w^+ m^{-1}) = \rho_{\beta,\tau}(w^-)$ is hence satisfied for all $w \in \pi_1(S)$.

We illustrate the correspondence between characteristic links and dihedral homomorphisms.

\begin{example}[The 2-component unlink]\label{ex:unlink}

Let $L$ denote the $2$-component unlink. Its group, $\pi_1(S^3\backslash L)\cong \Z \ast\Z$, is generated by the meridians, $m_1$ and $m_2$, of the two components, see Figure \ref{fig:unlink}. Since the group is free, mapping $m_1$ and $m_2$ to any pair of reflections in $D_n$ defines a homomorphism  $\pi_1(S^3\backslash L)\cong \Z \ast\Z\to D_n$ and an associated dihedral cover. We give an explicit description of the bijection between these covers and characteristic links for $L$, together with a preferred choice of conjugacy class of reflections.

Denote by $S$ the obvious annulus with boundary $L$ and let $\alpha$ be the (oriented) core curve of the annulus. We have $H_1(S;\mathbb Z) \cong \mathbb Z\langle [\alpha] \rangle$. The Seifert matrix is $V=[0]$, so $A=[0]$ as well. Thus, every class in $H_1(S;\mathbb Z_n)$ is characteristic for all integers $n\in \Z$. We can choose a HNN presentation for $\pi_1(E_L)$ as in Lemma \ref{lem:hnn} so that $m=m_1$. Let $\tau\in D_n$ be a preferred reflection.  An unknot bounding a disk in $S$ is a characteristic knot which defines the homomorphism $\pi_1(E_L)\to D_n$ with image $ \langle \tau \rangle \cong \mathbb Z_2 \subset D_n$. The class $k[\alpha]$ is represented by a characteristic link consisting of $k$ parallel copies of $\alpha$ on $S$. It defines the homomorphism $\pi_1(S^3\backslash L)\cong \Z \ast\Z\to D_n$ given by 
\[
\begin{aligned}
m_1 &\mapsto \tau,\\
m_2 &\mapsto \tau y^k,
\end{aligned}
\]
where we use the notation for $D_n$ given in Equation~\ref{eq:Dn}.

To see this, we will apply Equation~\ref{eq:rho}. Denote by $\lambda$ the Alexander dual of $\alpha$ and represent $\lambda$ by an embedded based loop disjoint from $S$. We then have $\pi_1(S^3\backslash S)\cong \Z\langle [ \lambda] \rangle$, so that $lk(\lambda, k\alpha)=k$. Orienting the two meridians of $L$ appropriately, we have $m_2=m_1\lambda$, see Figure~\ref{fig:unlink}. Thus, by definition, the dihedral homomorphism $\rho_{k\alpha,\tau}$ determined by the characteristic link $k\alpha$ and the reflection $\tau$ satisfies $\rho_{k\alpha, \tau}(m_1)=\tau$ and 
\[
\rho_{k\alpha,\tau}(m_2)=\tau y^{lk(\lambda, k\alpha)}= \tau y^k,
\]
as claimed.  This calculation also shows that, by varying $\beta$ and $\tau$, we can define a homomorphism mapping the pair $(m_1, m_2)$ to any pair of reflections in $D_n$.

\begin{figure}
    \centering
     \begin{overpic}[width=0.55\linewidth]{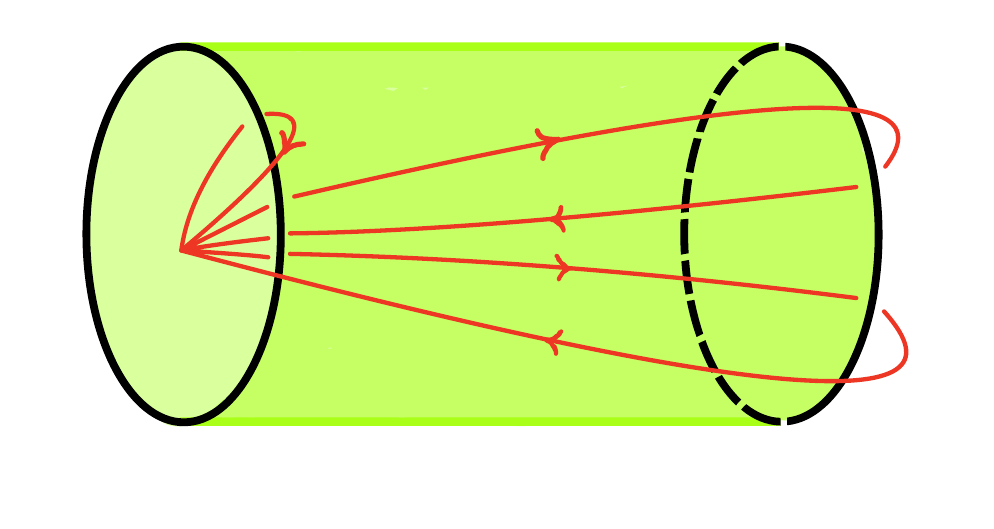}
        \put(65,43){\textcolor{red}{\Large $m_1$}}
        \put(14,35){\textcolor{red}{\Large $m_2$}}
        \put(65,13){\textcolor{red}{\Large $\lambda$}}
    \end{overpic}
    \caption{The three oriented curves $m_1$, $m_2$ and $\lambda$. Note that $m_2=m_1\lambda$}
    \label{fig:unlink}
\end{figure}

\end{example}

The following theorem generalizes the above example to all links in $S^3$. Its conclusion will be needed in our computations of dihedral cobordism groups.

\begin{theorem}\label{thm:recovery-of-dihedral-homs}
Let $L\subset S^3$ be an oriented link and $S$ a connected Seifert surface for $L$. For any representation $\rho \colon \pi_1(E_L) \rightarrow D_n$ sending meridians of $L$ to reflections in $D_n$, there exists a mod $n$ characteristic link $\beta \subset S^\circ$ and a preferred reflection $\tau \in D_n$ such that $\rho=\rho_{\beta,\tau}$.

\end{theorem}

To prove Theorem \ref{thm:recovery-of-dihedral-homs}, we need to introduce some technical tools. We start with the next proposition, which follows from the standard Mayer--Vietoris decomposition of $\Sigma_2(L)$, as described, for example, in Rolfsen's book \cite{R}. 

\begin{proposition}\label{prop:homology-of-double-cover}
Let $L\subset S^3$ denote a link with symmetrized Seifert matrix $A$ and double branched cover $\Sigma_2(L)$. There is an exact sequence
\[
0 \longrightarrow H_2(\Sigma_2(L);\mathbb Z)
\longrightarrow H_1(S;\mathbb Z)
\xrightarrow{\ A\ }
H_1(E_S;\mathbb Z)
\longrightarrow H_1(\Sigma_2(L);\mathbb Z)
\longrightarrow 0.
\]
In particular,
\[
H_1(\Sigma_2(L);\mathbb Z) \cong \coker(A). 
\]
Note that the matrix $A$ need not be nonsingular.
\end{proposition}

It is well known (see, for example,~\cite[p. 166]{CS}) that regular $2n$-fold dihedral branched covers of $L$ are also cyclic $n$-fold unbranched covers of $\Sigma_2(L)$. The latter are determined by homomorphisms $H_1(\Sigma_2(L);\mathbb Z)\rightarrow \mathbb{Z}_n$. Thus, it will be useful to explain how a mod~$n$ characteristic link $\beta$ for $L$ induces a group homomorphism $H_1(\Sigma_2(L);\mathbb Z)\rightarrow \mathbb{Z}_n$.

\begin{proposition}\label{prop:characteristic-link-gives-character}
Let $L \subset S^3$ be an oriented link, let $S$ be a connected oriented Seifert surface for $L$, and let $A$ be a symmetrized Seifert matrix of $S$. Let $\beta \subset S^{\circ}$ be a mod $n$ characteristic link for $(L,S)$. The function
\[
\lambda_{\beta} \colon H_1(E_S;\mathbb Z) \longrightarrow \mathbb Z_n,
\qquad
u \longmapsto lk(\beta,u) \pmod n,
\]
is a group homomorphism which vanishes on $A H_1(S;\mathbb Z)$. Consequently, it induces a homomorphism
\[
\chi_{\beta} \colon H_1(\Sigma_2(L);\mathbb Z) \longrightarrow \mathbb Z_n.
\]
Moreover, $\chi_{\beta}$ depends only on the reduction of $[\beta]\in H_1(S;\Z)$ modulo $n$, i.e. on the equivalence class of $\beta$.
\end{proposition}

\begin{proof}
    The first statement is a straightforward consequence of the definition of a mod $n$ characteristic link. Since the matrix $A$ is symmetric, we have:
\[
\lambda_{\beta}(Az)= [\beta]^t\, A\, z = (A[\beta])^t\, z \equiv 0 \pmod n.
\]
     Thus $\lambda_{\beta}$ vanishes on $A H_1(S;\mathbb Z)$ as claimed. By Proposition~\ref{prop:homology-of-double-cover},  $H_1(\Sigma_2(L);\mathbb Z) \cong H_1(E_S;\mathbb Z)/im(A)$, so the second statement follows. Finally, if $[\beta]\equiv[\beta']\mod n,$ then by linearity of linking in $S^3$ we have that $lk(\beta,u)\equiv lk(\beta', u)\mod n$, so $\chi_{\beta}=\chi_{\beta'}$.
\end{proof}

\begin{proposition}\label{prop:kernel-classifies-characters}
Let
\[
\bar A \colon H_1(S;\mathbb Z_n) \longrightarrow H_1(E_S;\mathbb Z_n)
\]
be the map represented by the matrix $A$ reduced mod $n$. (That is, if $\bar x\in H_1(S;\mathbb Z_n)$ is represented by
$x\in H_1(S;\mathbb Z)$, then $\bar A(\bar x)$ is the reduction of $A(x)$ mod $n$). There is an isomorphism
\[
\ker(\bar A) \cong \Hom(H_1(\Sigma_2(L);\mathbb Z),\mathbb Z_n)
\]
mapping any element $\bar\beta \in \ker(\bar A)$ to the character $\chi_{\beta}(x)$ defined in Proposition~\ref{prop:characteristic-link-gives-character}, where  $\beta \in H_1(S;\mathbb Z)$ is any integral lift of $\bar\beta$.
\end{proposition}

\begin{proof}
    By Proposition~\ref{prop:homology-of-double-cover}, there is a surjection
\[
q \colon H_1(E_S;\mathbb Z) \twoheadrightarrow H_1(\Sigma_2(L);\mathbb Z)
\]
with kernel $A H_1(S;\mathbb Z)$. Thus, composing  with $q$ identifies
$\Hom(H_1(\Sigma_2(L);\mathbb Z),\mathbb Z_n)$ with the subgroup of
$\Hom(H_1(E_S;\mathbb Z),\mathbb Z_n)$ consisting of those homomorphisms
which vanish on $A H_1(S;\mathbb Z)$.

Alexander duality identifies $H_1(S;\mathbb Z)$ with
$\Hom(H_1(E_S;\mathbb Z),\mathbb Z)$ via the pairing
\[
([\beta],u)\longmapsto lk(\beta^+,u),
\]
where $\beta^+$ is the positive normal push-off of $\beta$ on $S$ into
$E_S$. Reducing mod $n$ gives an isomorphism
\begin{equation} \label{eq:duality-Zn}
    H_1(S;\mathbb Z_n)\cong \Hom(H_1(E_S;\mathbb Z),\mathbb Z_n).
\end{equation}
Under this identification, an element $\bar\beta \in H_1(S;\mathbb Z_n)$
corresponds to the homomorphism
\[
\lambda_{\bar\beta} \colon H_1(E_S;\mathbb Z) \to \mathbb Z_n,
\qquad
u \mapsto lk(\beta^+,u) \pmod n,
\]
where $\beta \in H_1(S;\mathbb Z)$ is any integral lift of $\bar\beta$. It is straightforward to check that this is well-defined. 

If $\bar\beta \in \ker(\bar A)$ (equivalently, any of its integral lifts is a mod~$n$ characteristic knot), then Proposition~\ref{prop:characteristic-link-gives-character}
shows that $\lambda_{\bar\beta}$ vanishes on $A H_1(S;\mathbb Z)$, and hence
descends uniquely to a character
\[
\chi_\beta \in \Hom(H_1(\Sigma_2(L);\mathbb Z),\mathbb Z_n).
\]
This is the map in the theorem statement.

Conversely, let
\[
\chi \in \Hom(H_1(\Sigma_2(L);\mathbb Z),\mathbb Z_n).
\]
Then $\chi \circ q$ is a homomorphism $H_1(E_S;\mathbb Z)\to \mathbb Z_n$.
Under the identification above, there is a unique class
$\bar\beta \in H_1(S;\mathbb Z_n)$ corresponding to $\chi \circ q$.
Since $\chi \circ q$ vanishes on $A H_1(S;\mathbb Z)$, we have, $\forall z \in H_1(S;\mathbb Z)$,
\[
0 = (\chi \circ q)(Az) = \bar\beta^T A z = (\bar A \bar\beta)^T z.
\]
Since the Alexander-duality pairing is
nondegenerate, it follows that $\bar A \bar\beta=0$. Thus
$\bar\beta \in \ker(\bar A)$. We have thus constructed an inverse to the map $\ker(\bar A) \rightarrow \Hom(H_1(\Sigma_2(L);\mathbb Z),\mathbb Z_n)$ in the theorem statement, concluding the proof. 
\end{proof}

We are now ready for the proof of the main result of this section.

\begin{proof}[Proof of Theorem \ref{thm:recovery-of-dihedral-homs}]
Let
\[
\rho\colon \pi_1(E_L)\longrightarrow D_n
\]
send a meridian of each component of $L$ to a reflection and satisfy
$\rho(m)=\tau$ in the HNN presentation described in Lemma \ref{lem:hnn}. We show that there exists a mod $n$ characteristic link $\beta \subset S^\circ$ such that $\rho=\rho_{\beta,\tau}$. Let
\[
q\colon D_n\to D_n/\langle y\rangle\cong \mathbb Z_2
\]
be the quotient map (c.f. Equation \ref{eq:Dn}), and let
\[
\phi\colon \pi_1(E_L)\to\mathbb Z_2
\]
be given by mod $2$ intersection number with $S$. Thus $\phi$ sends every
meridian of $L$ to $1$ and vanishes on $\pi_1(E_S)$. Since $q\circ\rho$ and $\phi$ agree on meridians of $L$, and meridians generate $\pi_1(E_L),$ we have 
\[
q\circ\rho=\phi.
\]
It follows that
\[
\rho(\pi_1(E_S))\subset \ker(q)=\langle y\rangle.
\]
Since $\rho(\pi_1(E_S))\subset \langle y\rangle$ and $\langle y\rangle$ is
abelian, the restriction $\rho|_{\pi_1(E_S)}$ factors through a homomorphism
\[
\lambda\colon H_1(E_S;\mathbb Z)\to\mathbb Z_n \cong \langle y\rangle\subset D_n,
\]
where the isomorphism $\mathbb Z_n \to \langle y\rangle$ sends 1 to $y$. 
Thus, for every $w \in\pi_1(E_S)$,
\[
\rho(w)=y^{\lambda([w])}.
\]
For $w\in\pi_1(S)$, the HNN relation $mw^+m^{-1}=w^-$ gives
\[
\tau \,y^{\lambda([w^+])}\,\tau^{-1}=y^{\lambda([w^-])}.
\]
But the LHS also equals $y^{-\lambda([w^+])}$, which implies
\[
\lambda([w^+])+\lambda([w^-])\equiv 0\pmod n.
\]
Using $[w^+]+[w^-]=A[w]$, we get
\[
\lambda(A\, [w])=0.
\]
Hence $\lambda$ vanishes on $AH_1(S;\mathbb Z)$ and descends to a character
\[
\chi\colon H_1(\Sigma_2(L);\mathbb Z)\longrightarrow \mathbb Z_n.
\]
By Proposition~\ref{prop:kernel-classifies-characters}, $\chi=\chi_\beta$ for some
mod $n$ characteristic link $\beta\subset S^\circ$.
Since $\lambda$ and $\lambda_\beta$ (as defined in Proposition~\ref{prop:characteristic-link-gives-character}) are the pullbacks of the same character
$\chi=\chi_\beta$ under the natural map
\[
H_1(E_S;\mathbb Z)\twoheadrightarrow H_1(\Sigma_2(L);\mathbb Z),
\]
we have that $\lambda=\lambda_\beta$. Therefore
 $\rho$ and $\rho_{\beta,\tau}$ agree on $\pi_1(E_S)$; but they also both send $m$ to $\tau$. Hence $\rho=\rho_{\beta,\tau}$. This concludes the proof.
\end{proof}
\begin{remark}\label{rmk:rhobetaequivalence}
    Let $\beta,\beta' \subset S^\circ$ be mod $n$ characteristic links and let $\tau,\tau'$ be reflections. The $D_n$-covers determined by $(\beta,\tau)$ and $(\beta',\tau')$ are isomorphic, in the sense of Definition~\ref{def:Dncover}, if and only if
    \[[\beta']=\pm[\beta]\in H_1(S;\Z_n) \quad \text{and} \quad \tau,\tau' \text{ are conjugate in } D_n.\]
    First we check that inner automorphisms of $D_n$ change $\beta$ at most up to sign. Indeed, post-composing $\rho_{\beta,\tau}$ with conjugation by $y^j$ fixes $[\beta]$; and post-composing with  conjugation by $x$ replaces $[\beta]$ by $-[\beta]$. The sign is immaterial for the invariants of Section~\ref{sec:meat}, since $[\beta]^t A [\beta]$ is a quadratic form. If instead one only asks that $\rho_{\beta,\tau}$ and $\rho_{\beta',\tau'}$ be equivalent in the sense used earlier in this section, that is up to an arbitrary automorphism of $D_n$, then the condition becomes $[\beta']=k[\beta]$ for some $k \in (\Z/n)^\times$, and the invariants of Section~\ref{sec:meat} are not preserved. For the converse, it suffices to observe, again, that passing from $[\beta]$ to $-[\beta]$ is realized by post-composing $\rho_{\beta,\tau}$ with conjugation by $x$. 
\end{remark}

We conclude this subsection with two explicit examples of mod $n$ characteristic links.
\begin{example}[(2,n)-torus links]\label{ex:generator}
Let $T(2,n)$ be a positively oriented $(2,n)$-torus link. Denote $V_n$ the Seifert matrix described in Example \ref{ex:torustake1}. One can check that the primitive vector 
 \begin{equation*} \beta_n=
        \begin{pmatrix}
1 \\ 2\\ \vdots \\ n-1
\end{pmatrix}
    \end{equation*}
    satisfies the condition
     \begin{equation*} (V_n+ V_n^t)\, [\beta_n]=
        \begin{pmatrix}
-2 & 1 & 0 & 0 & \dots & 0 \\
1 & -2 & 1 & 0 & \dots & 0 \\
\vdots & \ddots & \ddots & \ddots & \ddots & \vdots \\
0 & \dots & \dots & \dots  & 1 & -2
\end{pmatrix}
\cdot    \begin{pmatrix}
1 \\ 2\\ \vdots \\ n-1
\end{pmatrix}
= 
  \begin{pmatrix}
0 \\ \vdots \\ 0 \\ -n
\end{pmatrix} 
\equiv 0 \pmod{n}
    \end{equation*}
and is hence a mod $n$ characteristic knot for $T(2,n)$. In particular, it defines a surjective representation $\rho_{\beta_n,\tau} \colon \pi_1(E_{T(2,n)})\rightarrow D_n$ for any choice of preferred reflection $\tau \in D_n$. Moreover, the self-pairing of $[\beta_n]$ via $A_n=V_n+V_n^t$ is
\[[\beta_n]^t \, A_n \, [\beta_n]=-n^2+n .\]

\begin{remark}\label{rmk:torus}
    Up to post-composing with an automorphism of $D_n$, there is a unique surjective representation $\rho_n \colon \pi_1(S^3 \setminus T(2,n)) \twoheadrightarrow D_n$. Indeed, any such $\rho_n$ is fully determined by the assignment of two reflections $\tau$ and $\tau y^k$ to the two parallel strands highlighted in the projection diagram of $T(2,n)$ on the left hand side of Figure \ref{fig:tori}. Moreover, surjectivity is achieved if and only if $\gcd(n,k)=1$. Notice that all such $\rho_n$'s are equivalent to the Fox coloring depicted on the right hand side of Figure \ref{fig:tori} via the automorphism
    \[D_n \longrightarrow D_n\]
    \[\tau \mapsto \tau, \qquad y \mapsto y^k.\]
    Since any automorphism of $D_n$ can also be applied to a split union of copies of $T(2,n)$ (and to any surface over which $\rho_n$ extends), we can conclude that the order of $[(T(2,n),\rho_n)]$ in any of the cobordism groups we introduced does not depend on the outer-equivalence class of $\rho_n$.
     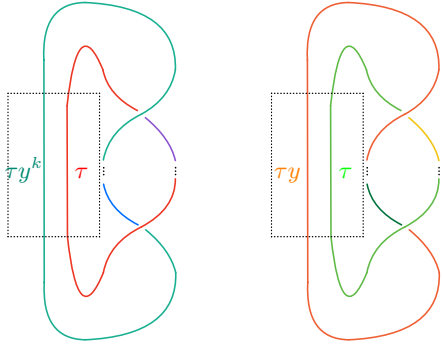
\begin{figure}
    \centering
     \begin{overpic}[width=0.45\linewidth]{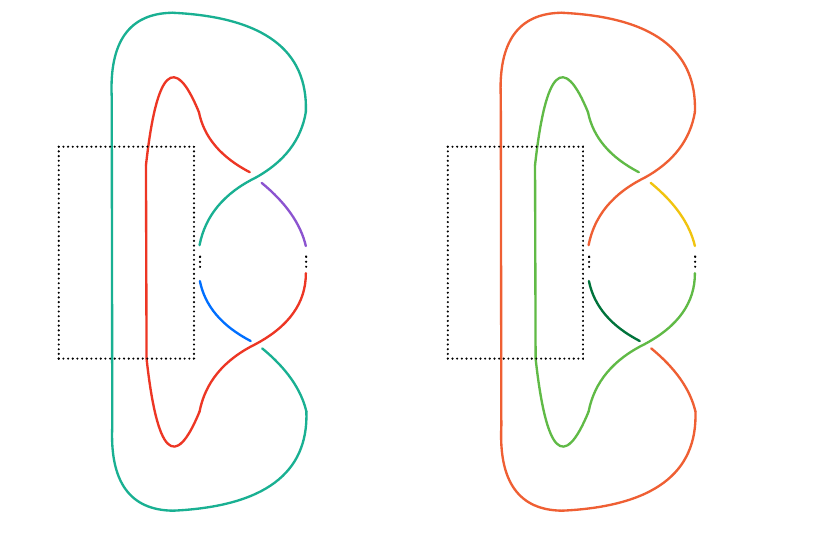}
        \put(6.5,32){\textcolor{PineGreen}{\small{$\tau y^k$}}}
        \put(19,32){\textcolor{red}{\small{$\tau$}}}
        \put(66,32){\textcolor{green}{\small{$\tau$}}}
        \put(54.5,32){\textcolor{orange}{\small{$\tau y$}}}
    \end{overpic}
    \caption{Two equivalent Fox colorings of the torus link $T(2,n)$, where $\gcd(n,k)=1$.}
    \label{fig:tori}
\end{figure}
\end{remark}

\end{example}

\begin{example}[Hopf links]\label{ex:Hopf}

The torus link $H=T(2,2)$ is the positive Hopf link, and it admits a nontrivial homomorphism to $D_n$ sending the meridians of its two components to \emph{distinct} reflections if and only if $n$ is even. Set $n=2m$ and let $S$ be the once-twisted annulus bounding $H$. Then $H_1(S;\Z)\cong \Z$ is generated by a simple loop $\alpha \subset S^\circ$ and $A=[-2]$, consistently with the convention of Example~\ref{ex:torustake1}, which for $n=2$ gives $V_2=[-1]$. It follows that
\[A \, [m\alpha]=-2m \equiv 0 \pmod{2m}\]
and hence $m[\alpha]$ is a mod $n$ characteristic link for $H$. By a computation exactly analogous to the one detailed in Example~\ref{ex:unlink}, the homomorphism determined by $(\beta, \tau)$ sends the meridians of the two components of $H$ to $\tau$ and $\tau y^m$ respectively. Moreover, we have that
\[[m\alpha]^t \, A \,[m\alpha]=-2m^2=-\frac{n^2}{2}.\]

Let $[H]$ denote the dihedral $n$-fold cover described by the mod $n$ characteristic link $m\alpha$ with a preferred reflection $\tau \in D_n$. We conclude this example by noticing that $2[H]=0\in Cob^{\leftrightarrow}_{D_n}(S^3)$. Indeed, it is shown in Figure \ref{fig:hopf} that the branch set of $2[H]$ becomes a $2$-component unlink after two colored band attachments. Each component of the resulting unlink is split and colored by a single reflection, so it can be capped off by a colored disk; this completes the null-cobordism.
 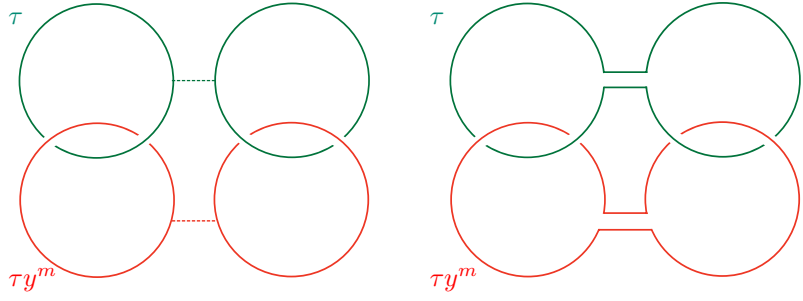
\begin{figure}
    \centering
     \begin{overpic}[width=0.75\linewidth]{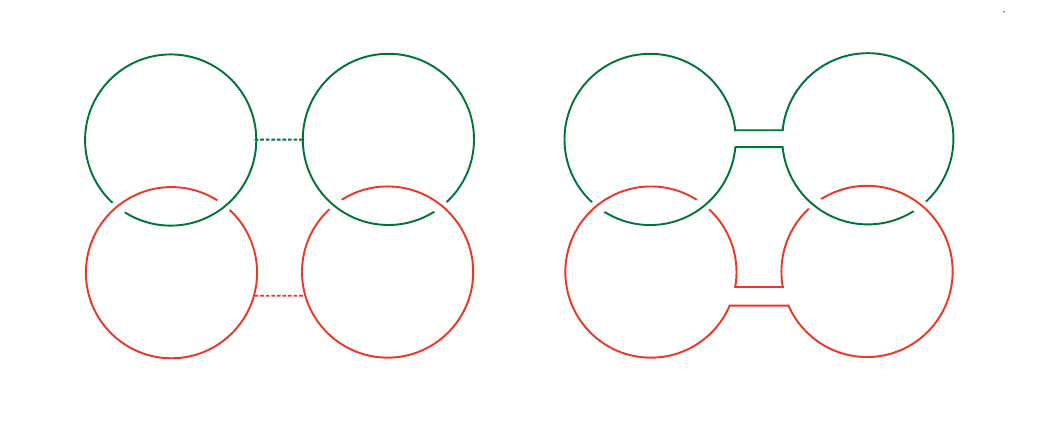}
        \put(7,35){\textcolor{PineGreen}{$\tau$}}
        \put(7,7){\textcolor{red}{$\tau y^m$}}
        \put(52,35){\textcolor{PineGreen}{$\tau$}}
        \put(52,7){\textcolor{red}{$\tau y^m$}}
       
    \end{overpic}
    \caption{A proof of the fact that $2H$ becomes a $2$-component unlink after two colored band attachments.}
    \label{fig:hopf}
\end{figure}

\end{example}

\subsection{Mod n characteristic surfaces}

In this section we introduce the notion of mod $n$ characteristic surface for an oriented properly embedded surface $F$ in $D^4$. 
We then show how to use characteristic surfaces to produce dihedral covers of $D^4$ branched over $F$. This section builds on work of Kjuchukova and Orr in \cite[Section 5.1]{KO} and we refer the reader to their paper for further details on this topic.

\begin{definition}
    Let $F\subset D^4$ be a properly embedded smooth oriented surface with non-empty boundary and let $H \subset D^4$ a connected Seifert solid for $F$ (see Definition \ref{def:ssolid}). 
    A properly embedded oriented surface $\Sigma \subset H$ 
is a \textit{mod $n$ characteristic surface for $F$} if for every class $w \in H_1(H; \Z)$ represented by a curve $\gamma \subset H$, the following condition holds
\[lk_{D^4} (i_+\gamma, \Sigma) + lk_{D^4}(i_-\gamma,\Sigma) \equiv 0 \pmod{n}.\]
Here $i_+\gamma$ and $i_- \gamma$ respectively denote the positive and negative push-offs of $\gamma$ inside $D^4 \setminus H$.

\end{definition}

\begin{proposition}\label{prop:charsurf}
    Let $F \subset D^4$ be a properly embedded smooth oriented surface with $\partial F\neq \emptyset$. Suppose that there exists a Seifert solid $H$ for $F$ containing a mod $n$ characteristic surface $\Sigma \subset H$. Then, for any reflection $\tau \in D_n$, there exists a homomorphism
    \[\rho_{\Sigma, \tau}\colon \pi_1(E_F)\longrightarrow D_n\]
    sending meridians of $F$ to reflections in $D_n$. 

    Let $L=F \cap \partial D^4$, $\beta=\Sigma \cap \partial D^4$ and $S=H \cap \partial D^4$. Then $\beta \subset S^\circ$ is a mod $n$ characteristic link for $L$ and the following diagram commutes
    \begin{center}
\begin{tikzcd}
\pi_1(E_L) \arrow[d, "\iota"'] \arrow[rr, "\rho_{\beta, \tau}"] &  & D_n \\
\pi_1(E_F) \arrow[rru, "\rho_{\Sigma, \tau}"']                  &  &    
\end{tikzcd}.
    \end{center}
    Here, $\iota \colon \pi_1(E_L) \rightarrow \pi_1(E_F)$ denotes the inclusion induced homomorphism and $\rho_{\beta, \tau}$ is defined as in Equation \ref{eq:rho}, where the HNN decompositions given by Lemma \ref{lem:hnn} and Lemma \ref{lem:hnnh} are chosen so that the map
     \[ \pi_1(E_L)\cong \bigl(\pi_1(E_S)*\Z \langle m \rangle / \sim\bigr) \; \xlongrightarrow{\iota} \bigl(\pi_1(E_H)*\Z \langle m \rangle / \sim\bigr) \; \cong \pi_1(E_F)\]
     sends $m \in \pi_1(E_S)$ to $m \in \pi_1(E_H)$.
\end{proposition}
\begin{proof}[Sketch of proof]
    By Lemma \ref{lem:hnnh}, $\pi_1(E_F)$ admits the following HNN presentation
    \[\pi_1(E_F)\cong \pi_1(E_H)* \Z \langle m \rangle  / \{mi_+(w) m^{-1}=i_-w \text{ for all } w \in \pi_1(H)\},\]
     where $H$ is a Seifert solid for $F$, $m$ is a generator of the infinite cyclic group and $i_+$, $i_-$ resp. denote the positive and negative push-offs of $w\in \pi_1(H)$. For a given reflection $\tau \in D_n$, we can construct a representation $\rho_{\Sigma,\tau}\colon \pi_1(E_F)\rightarrow D_n$ by setting
     \[\rho_{\Sigma, \tau}(m)=\tau, \qquad \rho_{\Sigma, \tau} (\gamma)=y^{lk(\Sigma,\gamma)} \quad \text{ for every } \gamma \in \pi_1(E_H),\]
     where $y$ is as in Equation \ref{eq:Dn}. As in the case of characteristic links, this is well-defined and the computation is identical. 

     The last part of the statement directly follows from the definitions of $\rho_{\beta, \tau}$ and $\rho_{\Sigma, \tau}$, together with the identity $lk_{D^4}(\Sigma,\gamma)=lk_{S^3}(\beta,\gamma)$ for every loop $\gamma \subset E_S$. 
\end{proof}

We conclude this section with some lemmas. We start with the following result, which will be used repeatedly in Section~\ref{sec:meat}.

\begin{lemma}\label{lem:glchar}
Let $L \subset S^3$ be a link, $\Sigma \subset S^3$ a spanning surface for $L$ and $\rho \colon \pi_1(E_L)\rightarrow D_n$ a homomorphism sending every meridian of $L$ to a reflection. Suppose that $\beta \subset \Sigma^\circ$ is a link such that
\[\rho(w)=y^{lk(\beta,w)} \quad \text{for every } w \in \pi_1(E_\Sigma).\]
If $U$ is a matrix representing the Gordon--Litherland pairing of $\Sigma$ with respect to some basis of $H_1(\Sigma;\Z)$, we have
\begin{equation}\label{eq:Utimesbeta}
    U \, [\beta]\equiv 0 \pmod n.
\end{equation}
In particular, if $\Sigma=S$ is an oriented Seifert surface for $L$, Equation~\ref{eq:Utimesbeta} becomes $A[\beta]\equiv 0 \pmod n$, that is, $\beta$ is a mod $n$ characteristic link for $(L,S)$.
\end{lemma}

\begin{proof}
Let $q\colon D_n \rightarrow D_n/\langle y \rangle \cong \Z_2$ be the quotient map, and let $\phi \colon \pi_1(E_L)\rightarrow \Z_2$ be given by the mod $2$ intersection number with $\Sigma$. The map $\phi$ is a well-defined homomorphism, as long as the base point of $\pi_1(E_L)$ is chosen away from $\Sigma$. Both $q \circ \rho$ and $\phi$ send every meridian of $L$ to the generator of $\Z_2$, and meridians generate $\pi_1(E_L)$; hence $q \circ \rho=\phi$. In particular, $\rho$ sends every loop meeting $\Sigma$ transversely in an odd number of points to a reflection, and every loop in $E_\Sigma$ to a power of $y$.

Since the Gordon--Litherland pairing $\mathcal U$ is bilinear and every class in $H_1(\Sigma;\Z)$ is a sum of classes represented by simple closed curves, it suffices to show that $[w]^t \, U \, [\beta] \equiv 0 \pmod{n}$ for any simple closed curve $w \subset \Sigma$. By definition and symmetry of $\mathcal{U}$, we have \[[w]^t \, U \, [\beta]=lk(\beta,\tau^{-1}(w)).\] Moreover, since $\tau^{-1}(w)\subset E_\Sigma$, by hypothesis $\rho$ evaluates on $\tau^{-1}(w)$ by linking with $\beta$.

Suppose first that $w$ is orientation-preserving in $\Sigma$, so that $\tau^{-1}(w)=w^+\sqcup w^-$ consists of the two push-offs of $w$ determined by the tubular neighbourhood of $w$ in $\Sigma$. There exists  a loop meeting $\Sigma$ transversely in a single point of $w$ such that the positive and negative pushoffs of $w$ are satisfy the following relation in $\pi_1(E_L)$:
\[
w^-=\gamma\, w^+ \gamma^{-1}.
\]
By what we established earlier, $\rho(\gamma)$ is a reflection, so
\[y^{lk(\beta,w^-)}=\rho(w^-)=\rho(\gamma)\,y^{lk(\beta,w^+)}\,\rho(\gamma)^{-1}=y^{-lk(\beta,w^+)},\]
which gives \[lk(\beta,\tau^{-1}(w))=lk(\beta,w^+)+lk(\beta,w^-)\equiv 0 \pmod n.\]

Now suppose that $w$ is orientation-reversing, so that $\tau^{-1}(w)$ is a single circle $\widetilde w$ doubly covering $w$. Fix a point $a \in w$, set $\tau^{-1}(a)=\{a^+,a^-\}$, denote by $\ell$ the arc in $\widetilde w$ running from $a^+$ to $a^-$, and by $f$ a fiber arc from $a^-$ to $a^+$ meeting $\Sigma$ once, at $a$. The loop $c=\ell \cdot f$ meets $\Sigma$ transversely in a single point, so $\rho(c)$ is a reflection. On the other hand, $c^2$ is freely homotopic to $\widetilde w$ in $E_L$, so we have 
\[y^{lk(\beta,\tau^{-1}(w))}=\rho(\widetilde w)=\rho(c)^2=1,\]
that is, \[lk(\beta,\tau^{-1}(w))\equiv 0 \pmod n.\]

Finally, if $\Sigma=S$ is an oriented Seifert surface, then \[[w]^t \, U \,[\beta]=[w]^t\,A\,[\beta]\] for every $[w] \in H_1(S;\Z)$ by Remark~\ref{rmk:glandseifert}, so $A\,[\beta]\equiv 0 \pmod n$.
\end{proof}

\begin{lemma}\label{lem:bandsum}
   Let $\beta \subset S^\circ$ be a mod $n$ characteristic link for a Fox colored link $L \subset S^3$. If $L'\subset S^3$ is a Fox colored link obtained from $L$ by a (possibly non-orientable) colored band attachment via a band $b$ which does not intersect $S$ in its interior, then $S':=S\cup b$ is a spanning surface for $L'$, $\beta \subset S'^\circ$ satisfies
   \[U \, [\beta] \equiv 0 \pmod{n}\] and the representation given by Fox coloring of $L'$ coincides with
   \[w \mapsto y^{lk(\beta,w)} \quad \text{for every } w \in \pi_1(E_{S'}).\]
   If $b$ is an orientable band attachment and $S'$ is a Seifert surface for $L'$ with symmetrized Seifert matrix $A'$, then $\beta \subset S'^\circ$ is a mod $n$ characteristic link for $L'$ and
   \[[\beta]^t \, A \, [\beta]=[\beta]^t \, A'\, [\beta] \pmod{n^2}.\]
\end{lemma}

\begin{proof}
Since $S$ is connected, a Mayer--Vietoris argument shows that there is an identification $H_1(S';\Z) \cong H_1(S;\Z) \oplus \Z \langle [\gamma] \rangle$, where $\gamma \subset S'$ is a simple loop obtained as the union of the core of the band $b$ and an arc whose interior lies in $S^\circ$. Moreover, we have that $\pi_1(E_{S'}) \cong \pi_1(E_S)*\Z$, where the free $\Z$-summand is generated by the Alexander dual to $\gamma$ and can be represented by an embedded loop $\gamma^*$ near one component of the attaching region to $b$, so $\gamma^*$ is the product of two meridians of $L'$ which have the same Fox coloring, so the image of $\gamma^*$  in $D_n$ is trivial. (In more detail, we can write  $\gamma^*$ as the product of two based meridians of $L,$ say $\gamma^*=m_1m_2$, where $m_1$ and $m_2$ are meridians of strands in $L_1$ which share an endpoint with the same strand of $L$. Since the band attachment is colored, it follows that the Fox colorings of $m_1$ and $m_2$ agree. Therefore, the induced Fox coloring on $L'$ sends $\gamma^\ast$ to the square of a reflection in $D_n$.) This means that the Fox coloring of $L'$ restricted to $\pi_1(E_{S'})$ satisfies
\[ w \mapsto y^{lk(w, \beta)} \quad \text{for every } w \in \pi_1(E_{S'}).\]
Lemma~\ref{lem:glchar} then implies that $\beta \subset S'$ is such that $U\, [\beta] \equiv 0 \bmod{n}$, as claimed. The last part of the statement follows from Remark \ref{rmk:glandseifert}, since $U \, [\beta]=A'\, [\beta] \equiv 0 \pmod{n}$.
\end{proof}

We conclude this section with a lemma describing how a mod $n$ characteristic link is carried along a tubing, or through a half-twisted band attachment (see Section~\ref{sec:preliminaries} for a description of these moves). It will be used to compare invariants computed from different spanning surfaces.  

\begin{lemma}\label{lem:transport}
Let $L$, $\Sigma$, $\rho$ and $\beta \subset \Sigma^\circ$ be as in Lemma~\ref{lem:glchar}, and let $\Sigma'$ be obtained from $\Sigma$ either by a tubing or by a half-twisted band attachment. Then there exists a link $\beta' \subset \Sigma'^\circ$ with the property
\[\rho(w)=y^{lk(\beta',w)} \quad \text{for every } w \in \pi_1(E_{\Sigma'}),\]
and such that
\[[\beta']^t\, U'\, [\beta']=[\beta]^t \,U\, [\beta],\]
where $U$ and $U'$ denote Gordon--Litherland matrices of $\Sigma$ and $\Sigma'$ respectively. For a tubing, $\beta'$ can be chosen to be the union of $\beta$ with a certain number of parallel copies of the belt circle of the tube; for a half-twisted band attachment, we let $\beta'=\beta$.
\end{lemma}

\begin{proof}
Assume $\Sigma'$ is obtained from $\Sigma$ by a half-twisted band attachment and set $\beta':=\beta \subset \Sigma'$. Since a tubular neighbourhood of $\beta$ in $\Sigma$ is untouched by this operation, it follows that
\[[\beta']^t \, U' \, [\beta']=[\beta]^t \, U \, [\beta].\] Moreover, it is easy to check that the restriction of $\rho$ to $\pi_1(E_{\Sigma'})$ is given by $w \mapsto y^{lk(\beta',w)}$, as wanted, since $\beta$ does not link the meridian of the new band.

Suppose now that $\Sigma'$ is obtained by adding a tube $T$ to $\Sigma$ along two open disks $D_1^\circ, D_2^\circ \subset \Sigma^\circ$ disjoint from a neighbourhood of $\beta$. Denote by $e \subset \Sigma'^\circ$ the belt circle of $T$. 
Let $\gamma \subset E_{\Sigma'}$ be a based loop which enters $T$ through $D_1$, runs through the tube, exits through $D_2$ and closes up outside; see Figure \ref{fig:sum}. (We do not fret about the choice of paths to the basepoint because they do not affect what follows.)
 \begin{figure}
    \centering
     \begin{overpic}[width=0.55\linewidth]{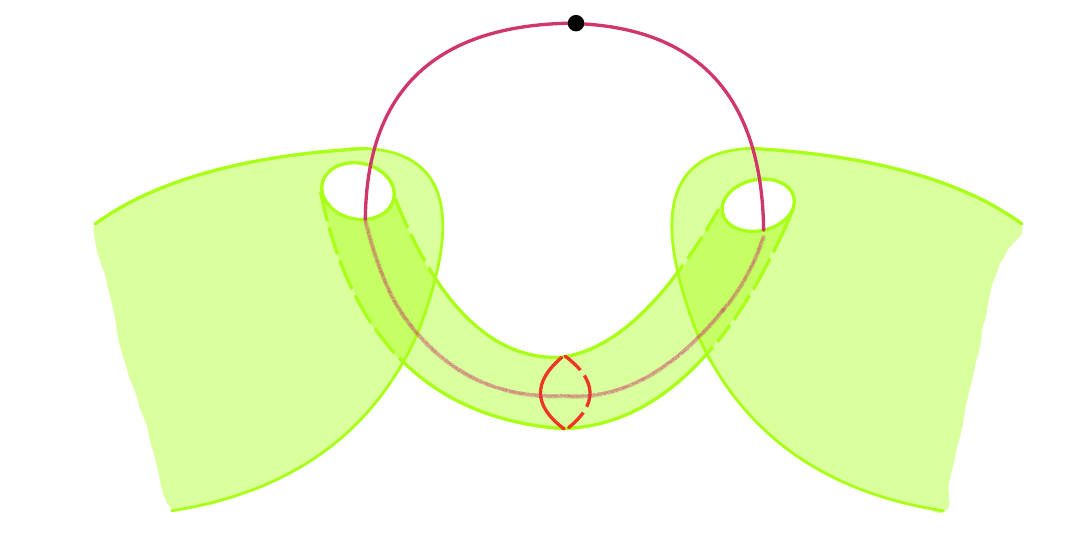}
       \put(58,43){\textcolor{red}{\large{$\gamma$}}}
       \put(2,42){\textcolor{ForestGreen}{\small{$(\Sigma \setminus (D_1^\circ \sqcup D_2^\circ))\cup T$}}}
       
       \put(45,7.5){\textcolor{red}{\large{$e$}}}
    \end{overpic}
    \caption{The curve $\gamma$.}
    \label{fig:sum}
\end{figure}
Orient $\gamma$ so that $lk(e,\gamma)=1$. 
Both \[w \mapsto \rho(w) \quad  \text{and} \quad w \mapsto y^{lk(\beta,w)}\] define homomorphisms \[H_1(E_{\Sigma'};\Z)\rightarrow \langle y \rangle\subset D_n.\] Indeed, $\Sigma$ and $\Sigma'$ represent the same class in $H_2(S^3,L;\Z_2)$. This implies that every loop in $E_{\Sigma'}$ has trivial mod $2$ intersection number with $\Sigma$, and $\rho$ maps it to a power of $y$ by the first paragraph of the proof of Lemma~\ref{lem:glchar}. The two homomorphisms agree on loops disjoint from $\Sigma\cup T$, because such loops lie in $E_\Sigma$. Define $c \in \Z$ by $\rho(\gamma)=y^{lk(\beta,\gamma)+c}$. The homomorphism $u \mapsto y^{lk(e,u)}$ vanishes on based loops disjoint from $\Sigma \cup T$, since such loops are disjoint from the disk $D_1$ spanning $e$ (since $D_1$ and the basepoint lie in separate connected components of $E_{\Sigma\cup T}$), and takes the value $1$ on $\gamma$. Since $H_1(E_{\Sigma'};\Z)$ is generated by $[\gamma]$ together with classes of loops disjoint from $\Sigma \cup T$, we conclude that
\[\rho(w)=y^{lk(\beta,w)+c\,lk(e,w)}=y^{lk(\beta',w)} \quad \text{for every } w \in \pi_1(E_{\Sigma'}),\]
where $\beta'\subset \Sigma'^\circ$ denotes the union of $\beta$ with $|c|$ parallel copies of the belt circle $e$, oriented as $e$ if $c>0$ and oppositely if $c<0$. 

It remains to compare the self-pairings. By bilinearity and symmetry of the Gordon--Litherland pairing (Remark~\ref{rmk:glandseifert}), we have
\[[\beta']^t\,U'\,[\beta']=[\beta]^t\,U'\,[\beta]+2c\; [\beta]^t \, U' \, [e]+c^2\; [e]^t \, U' \, [e]\]
A neighbourhood of $\beta$ in the surface is unchanged by the tubing, so \[[\beta]^t\,U'\,[\beta]=[\beta]^t\,U\,[\beta].\] Moreover, 
\[[e]^t \, U' \, [e]=lk(e,\tau'^{-1}(e))=0, \qquad [\beta]^t \, U' \, [e]=lk(\beta, \tau'^{-1}(e))=2\,lk(\beta,e)=0\]
and the conclusion follows.
\end{proof}

\section{Extending dihedral covers over the 4-ball}\label{sec:extensions}

In this section, we show how to extend a given $D_n$-cover over the $4$-ball, provided that such an extension exists.

As usual, let $L \subset S^3$ be an oriented link, $S \subset S^3$ a connected oriented Seifert surface for $L$, and $E_L$ and $E_S$ the exteriors of $L$ and $S$, respectively. Choose a basis of $H_1(S;\mathbb Z)$, let $V$ be the corresponding Seifert matrix of $S$ and let $A = V+V^T$.

Given a mod $n$ characteristic link $\beta$ for $L$, we have that
\[A\, [\beta]=n\, \alpha\]
for some class $\alpha \in H_1(E_S;\Z)$. It follows that
\begin{equation}\label{eq:inv}
    [\beta]^t\, A\, [\beta]=n\,[\beta]^t\, \alpha=a_1 n^2+a_2n
\end{equation}
for some integers $a_1, a_2$ with $0 \leq a_2 \leq n-1$.

The next proposition gives a sufficient condition for the branched cover associated to $\beta$ to bound in $Cob_{D_n}(S^3)$. Such condition is expressed in terms of the coefficients $a_1$ and $a_2$ in Equation \ref{eq:inv}.

\begin{proposition}\label{prop:inj}
Let $p$ be a dihedral $n$-fold cover branched along $L$ and $\beta \subset S^\circ$ a mod $n$ characteristic link such that $\omega_p=\rho_{\beta, \tau}$ for some reflection $\tau \in D_n$ (c.f. Theorem \ref{thm:recovery-of-dihedral-homs}). If $a_1$ is even and $a_2=0$, then \[[p]=0\in Cob_{D_n}(S^3).\]
\end{proposition}
\begin{proof}
  Let $S \subset S^3$ be the oriented Seifert surface containing the mod $n$ characteristic link $\beta$. Following \cite[Section 6.2]{KO}, we construct a Seifert solid $H \subset D^4$ whose boundary contains a Seifert surface for $L$, together with a mod $n$ characteristic surface $\Sigma \subset H$ meeting $\partial D^4$ along the mod $n$ characteristic link $\beta$ inducing the given cover $p$. 

  \textbf{Step 1:} One can show that, up to stabilizing $S$ a finite number of times, we can assume without loss of generality that $[\beta]^t \, A \, [\beta] =0$. The proof is identical to the one of Lemma \cite[Lemma 6.2]{KO} and is hence omitted. 

  \textbf{Step 2:} As in \cite[Section 6.2]{KO}, one can construct a Seifert solid $H \subset D^4$ such that $\partial H = F \cup_L -S$ for some properly embedded surface $F \subset D^4$, together with a characteristic surface $\Sigma \subset H$ satisfying $\Sigma \cap \partial D^4=\beta$. By Proposition \ref{prop:charsurf}, the pair $(\Sigma, \tau)$ defines a dihedral $n$-fold branched cover of $D^4$ branched along $F$, which over $\partial D^4 = S^3$ restricts to the cover corresponding to $(\beta, \tau)$.
\end{proof}
\begin{remark}\label{rem:knots}
   Let $\beta$ be a mod $n$ characteristic link and suppose that $[\beta]^t\, A\, [\beta]=a_1 n^2$, i.e. $a_2=0$ in the notation used in Equation \ref{eq:inv}. Then
    \[a_1 n^2=[\beta]^t\, A \, [\beta]= 2\,[\beta]^t\, V \, [\beta].\]
    If $n$ is odd, this implies that $a_1$ is even. In other words, when $n$ is odd the condition $a_2=0$ is enough to perform the construction in Proposition \ref{prop:inj}.
\end{remark}

On the other hand, we show that the vanishing of $a_2$ is a sufficient condition to have $[p]=0\in Cob^{\leftrightarrow}_{D_n}(S^3)$ in the unoriented realm.

\begin{proposition}\label{prop:injnon}
    Let $p$ be a $D_n$-cover branched along $L$ and $\beta \subset S^\circ$ a mod $n$ characteristic link such that $\omega_p=\rho_{\beta, \tau}$ for some reflection $\tau \in D_n$ (c.f. Theorem \ref{thm:recovery-of-dihedral-homs}).  If $a_2=0$ for some fixed orientation of $B_p$, then \[[p]=0\in Cob^{\leftrightarrow}_{D_n}(S^3).\]
\end{proposition}
\begin{proof}
If $a_2=0$ and $a_1$ is even, Proposition \ref{prop:inj} implies that $[p]=0\in Cob_{D_n}(S^3)$. In particular, this means that $[p]=0\in Cob^{\leftrightarrow}_{D_n}(S^3)$ as well.

Suppose now that $a_1$ is odd. This implies that $n$ is even, see Remark \ref{rem:knots}. Define $q$ to be the dihedral $n$-fold cover of $S^3$ branched over the split union of $B_p$ with two Hopf links $2H$ endowed with the Fox coloring determined by the characteristic link $n/2[\alpha]$ described in Example \ref{ex:Hopf} together with the choice of a reflection $\tau \subset D_n$. An orientable Seifert surface $S_q$ for $B_q$ can be obtained by tubing two orientable Seifert surfaces for $B_p$ and $B_{2H}$ respectively. With respect to this Seifert surface, a mod $n$ characteristic link $\beta_q$ for $q$ can be constructed as in Lemma \ref{lem:transport}. In particular, if $A_q$ is a symmetrized Seifert matrix for $S_q$, we have that
\[\beta_q^t A_q \beta_q=\beta^t A \beta -n^2=(a_1-1)n^2.\]
Proposition \ref{prop:inj} implies that $[q]=0\in Cob_{D_n}(S^3)$. Since $2[H]=0\in Cob^{\leftrightarrow}_{D_n}(S^3)$ by Example \ref{ex:Hopf} and $[q]=[p]+2[H]$, the conclusion follows.
\end{proof}

\begin{remark}\label{rmk:inj}
   What we really showed in Proposition \ref{prop:injnon} is that $p$ bounds an \textit{orientable} Fox colored surface in $D^4$, a stronger conclusion than merely $[p]=0\in Cob^{\leftrightarrow}_{D_n}(S^3)$. This will be used in the proof of Corollary~\ref{cor:cob}.
\end{remark}

\section{Cobordism invariants}\label{sec:meat}

The aim of this section is to define the invariant $\kappa_n$ of Theorem~\ref{thm:main} and to show that it is indeed a cobordism invariant, and defines the asserted group homomorphisms.

As usual, let $L \subset S^3$ be a link, $S$ an oriented connected Seifert surface, $V$ a Seifert matrix and $A=V+V^t$. Suppose that $\beta$ is a mod $n$ characteristic link for $L$, i.e. $A\,[\beta]\equiv 0\pmod n$. We can write
\[[\beta]^t \,A\,[\beta]=a_1n^2+a_2n\]
with $0 \leq a_2 \leq n-1$, see Equation \ref{eq:inv}. We set

\begin{equation}\label{eq:kappa-def}
\kappa_n(\beta):=\tfrac{1}{n}\,[\beta]^t\,A\,[\beta] \pmod 2n \;\in\; \Z_{2n}.
\end{equation}
By definition, $\kappa_n(\beta)=a_1n+a_2 \bmod 2n$. Denote by $[a_1]_2$ the reduction of $a_1$ modulo $2$ and observe that $\kappa_n(\beta)$ determines the pair $([a_1]_2, a_2)$ via
\[[a_1]_2=\frac{\kappa_n([\beta])-a_2}{n} \pmod{2} \quad \text{and} \quad a_2=\kappa_n([\beta]) \pmod n .\]

Let $Br_{D_n}(S^3)$ denote the set of $D_n$-covers of $S^3$ branched over an oriented link. The set underlying $Cob_{D_n}(S^3)$ is obtained from $Br_{D_n}(S^3)$ by identifying cobordant covers. 
\begin{lemma}\label{lem:invariant}
  For every integer $n \geq 2$, there is a well-defined map (which we also denote by $\kappa_n$)
\[\kappa_n \colon Br_{D_n}(S^3)\longrightarrow \Z_{2n}\]
  given by 
  \[p \mapsto \kappa_n([\beta]),\]
where $p$ is a $D_n$-cover branched over an oriented link $L$ and $\beta$ is a mod~$n$ characteristic link for $L$ such that the monodromy map $\omega_p$ inducing $p$ equals $\rho_{\beta,\tau}$ for some reflection $\tau \in D_n$.
\end{lemma}

\begin{proof}
Given a $D_n$-covering $p: Y\to S^3$ as above, by Theorem~\ref{thm:recovery-of-dihedral-homs}, we can associate to $p$ a mod $n$ characteristic link $\beta \subset S^\circ$ such that $\omega_p=\rho_{\beta, \tau}$ for some reflection $\tau \in D_n$, as asserted in the Lemma. We then have to show that the definition of $\kappa_n$ does not depend on the choice of the mod $n$ characteristic link $\beta$ and Seifert surface $S$.

Suppose that $\beta$ and $\beta'$ are two mod $n$ characteristic links associated to the same $D_n$-cover $p$. If $\beta$ and $\beta'$ are contained in the same Seifert surface $S$, then $[\beta']=\pm[\beta]\in H_1(S;\Z_n)$ by Remark~\ref{rmk:rhobetaequivalence}. Replacing $\beta'$ by $-\beta'$ if necessary, which changes neither $[\beta']^t\,A\,[\beta']$ nor the $D_n$-cover it induces, we may assume that \[[\beta']=[\beta]+n\,x \quad  \text{for some } x \in H_1(S;\Z).\] If
\[[\beta]^t \,A\,[\beta]=a_1n^2+a_2 n\]
with $0 \leq a_2 \leq n-1$, then
\[[\beta']^t\, A\,[\beta']=[\beta]^t\,A\,[\beta]+2n\, x^t\,A\,[\beta]+n^2 \, x^t\,A\,x.\]
Since $\beta$ is a mod $n$ characteristic knot, we can write $A\,[\beta]=n\,y$ for some $y \in H_1(E_S;\Z)$. Moreover, \[x^t \,A\, x=2 \,x^t\,V\,x\] and hence
\[[\beta']^t \,A\, [\beta']=(a_1+2 x^t y + 2 x^t V x) n^2 +a_2 n=a_1'n^2+a_2'n,\]
where $a_2'=a_2$ and $a_1'$ has the same parity of $a_1$. Dividing by $n$, we get \[\kappa_n(\beta')=\kappa_n(\beta)+2n\left(x^ty+x^tVx\right)\equiv \kappa_n(\beta) \in \Z_{2n}.\] We can conclude that the values of $\kappa'_n$ coincide for these two mod $n$ characteristic links.

To finish the proof, we have to deal with the case when $\beta$ and $\beta'$ are contained in non-isotopic Seifert surfaces. To do this, recall that any two oriented connected Seifert surfaces for the same link become ambiently isotopic after finitely many stabilizations, see Lemma \ref{lem:stab}. It suffices to show that, given $\beta \subset S^\circ$ and a Seifert surface $S^s$ obtained from $S$ via a single stabilization, there exists a mod $n$ characteristic link $\beta^s$ in the interior of $S^s$ such that $\beta^s$ determines the monodromy of $p$ and
\[[\beta^s]^t\, A^s\, [\beta^s]=[\beta]^t \,A \,[\beta],\]
where $A^s$ denotes a symmetrized Seifert matrix for $S^s$. 

After possibly performing an isotopy of $\beta$ inside $S$, we may assume that the attaching disks of the stabilizing tube are disjoint from $\beta$. 
Lemma~\ref{lem:transport}, applied to $\rho=\rho_{\beta,\tau}$ and to the tubing producing $S^s$, provides a link $\beta^s \subset (S^s)^\circ$ such that \[\rho(w)=y^{lk(\beta^s,w)}\quad \text{for every } w \in \pi_1(E_{S^s})\] and
\[[\beta^s]^t \,A^s \,[\beta^s]=[\beta]^t \,A\, [\beta].\] More precisely, $\beta^s$ is the union of $\beta$ with $c$ parallel copies of the belt circle of the stabilizing tube (where $c$ is as in the proof of Lemma~\ref{lem:transport}).  Lemma~\ref{lem:glchar} applied to  $\beta^s$ then implies that $A^s\, [\beta^s]\equiv 0 \pmod n$, i.e. $\beta^s$ is a mod $n$ characteristic link for $L$ in $S^s$. Finally, we check that $\beta^s$ and $\beta$ define the same $D_n$-cover. Let $\mu$ be the meridian of $L$ appearing in the HNN presentation of $\pi_1(E_L)$ determined by $S^s$, as in Lemma~\ref{lem:hnn}, and set $\tau''=\rho(\mu)$, a reflection by the first paragraph of the proof of Lemma~\ref{lem:glchar}. The representations $\rho$ and $\rho_{\beta^s,\tau''}$ agree on $\pi_1(E_{S^s})$ and on $\mu$, which together generate $\pi_1(E_L)$ by Lemma~\ref{lem:hnn}; hence $\rho_{\beta^s,\tau''}=\rho=\rho_{\beta,\tau}$.

Using what we have just shown, we can always reduce to the case when $\beta$ and $\beta'$ are contained in the same Seifert surface and the conclusion follows.
\end{proof}

We now perform a similar construction in the unoriented context. Let $\hat Br_{D_n}(S^3)$ denote the set of $D_n$-covers of $S^3$ branched over an unoriented link. Notice that $Cob^{\leftrightarrow}_{D_n}(S^3)$ is obtained from $\hat Br_{D_n}(S^3)$ by identifying cobordant covers. 

\begin{lemma}\label{lem:nonorinv}
For every integer $n \geq 2$, there is a well-defined map 
\[\kappa^{\leftrightarrow}_n \colon \hat Br_{D_n}(S^3)\longrightarrow \Z_{n}\]
  given by 
  \[p \mapsto \kappa_n([\beta]) \pmod{n},\]
where $p$ is a $D_n$-cover branched over an unoriented link $L$ and $\beta$ is a mod~$n$ characteristic link for $L$ such that the monodromy map $\omega_p$ inducing $p$ equals $\rho_{\beta,\tau}$ for some reflection $\tau \in D_n$.
\end{lemma}
\begin{proof}
    As in the proof of Lemma \ref{lem:invariant}, we have by Theorem~\ref{thm:recovery-of-dihedral-homs} that such $(\beta, \tau)$ exist; and we need to show that the definition of $\kappa^{\leftrightarrow}_n$ does not depend on the choice of the mod $n$ characteristic link $\beta$ and Seifert surface.

    Let $\beta\subset S^\circ $ and $\beta' \subset S'^\circ$ be mod $n$ characteristic links for $L \subset S^3$ corresponding to equivalent $D_n$-covers for some choices of reflections $\tau$ and $\tau'$ respectively. Here $S$ and $S'$ are oriented Seifert surfaces inducing possibly different orientations on $L$.  
    
    By Lemma \ref{lem:spanstab}, there exists a possibly non-orientable surface $\Sigma$ which is obtained from both $S$ and $S'$ after a finite number of tubings and additions of half twisted bands. Lemma \ref{lem:transport} implies the existence of links $\beta^s, \beta'^s \subset \Sigma^\circ$ satisfying
    \[ y^{lk(\beta^s,w)}=\rho_{\beta,\tau}(w)=\rho_{\beta'^s,\tau'}(w)=y^{lk(\beta'^s,w)} \quad \text{for every }  w \in \pi_1(E_\Sigma).\]
    In particular, the above equation implies that 
    \[
    lk(\beta^s,w)=lk(\beta'^s,w) \pmod n  \quad \text{ for every } w\in \pi_1(E_\Sigma).
    \]
Since the isomorphism in Equation~\ref{eq:duality-Zn} is given by linking with a class in $H_1(S; \Z_n),$ we conclude that
    \[[\beta^s]=[\beta'^s]\in H_1(\Sigma;\Z_n),\]
    and therefore
    \[[\beta^s]^t \, U \, [\beta^s] \equiv [\beta'^s]^t \, U \, [\beta'^s] \pmod{n^2}  ,\]  where $U$ denotes the Gordon--Litherland matrix for $\Sigma$. Lemma \ref{lem:transport} (with Remark \ref{rmk:glandseifert}) also implies that
    \[[\beta^s]^t \, U \, [\beta^s]=[\beta]^t \, A \, [\beta], \qquad [\beta'^s]^t \, U \, [\beta'^s]=[\beta']^t \, A' \, [\beta'].\]
Putting everything together, we get that
\[[\beta]^t \, A [\beta] \equiv [\beta']^t \, A' \, [\beta'] \pmod{n^2},\]
as desired.
\end{proof}

We now want to promote the map $\kappa_n$ in Lemma \ref{lem:invariant} to a cobordism invariant. This is done in the following proposition.

\begin{proposition}\label{prop:cobordinv}
For every integer $n \geq 2$, the map $\kappa_n$ from Lemma \ref{lem:invariant} induces a group homomorphism
\[\kappa_n \colon Cob_{D_n}(S^3) \longrightarrow \Z_{2n}.\]

\end{proposition}
\begin{proof} 
Suppose that $p_0$ and $p_1$ are cobordant and let $C\subset S^3\times I$ be a properly embedded smooth surface representing the colored branch set of a cobordism between $p_0$ and $p_1$. By considering an embedded handle decomposition of $C$ relative to $\partial S^3 \times I$, we can deduce that $C$ is obtained from the colored branch set $B_{p_0}$ of $p_0$ by a finite sequence of embedded $2$-dimensional handle attachments. In the following, we show that $\kappa_n$ is invariant under elementary colored cobordisms, i.e. under embedded colored $0$-, $1-$ and $2$- handle attachments. Since a colored $0$-handle attachment is just a colored $2$-handle attachment viewed upside down, we just need to show invariance of $\kappa_n$ under colored $1$- and $2$-handle attachments.

Let $L\subset S^3$ be a colored link and let $L'\subset S^3$ be obtained from $L$ via an embedded colored $1$-handle attachment in $S^3\times I$. This means that we can get $L'$ by gluing a colored oriented band $b$ to $L$ in $S^3$. We now claim that there exists a connected Seifert surface $S$ for the oriented link $L$ which is disjoint from the band $b$. Indeed, if $S$ is any connected Seifert surface for $L$, then $b$ possibly intersects $S$, and, after possibly performing a small ambient isotopy, we may assume that $b \cap S$ consists of a finite number of ribbon intersections in the interior of $S$. 
We number these ribbon intersections following the band $b$ from one component of its attaching region to the other and we denote them by $r_1, r_2, \dots, r_k$. We then add a new ribbon intersection $r$ between the first attaching region of $b$ and the first component, $r_1$, of $b \cap S^\circ$, by perturbing $b$ via an ambient isotopy of $S^3$ as in Figure \ref{fig:kink}.
 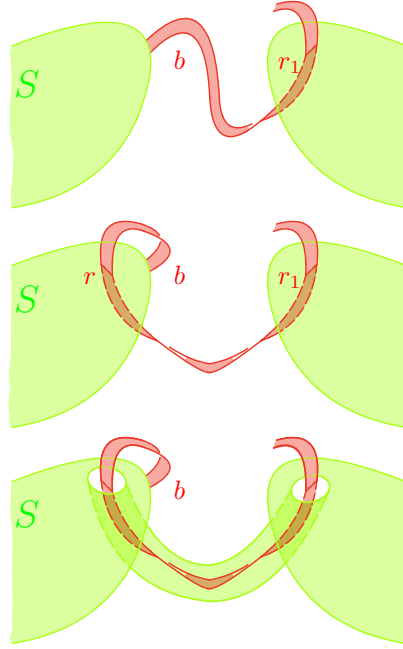
\begin{figure}[!htbp]
    \centering
    \begin{overpic}[width=0.4\linewidth]
     {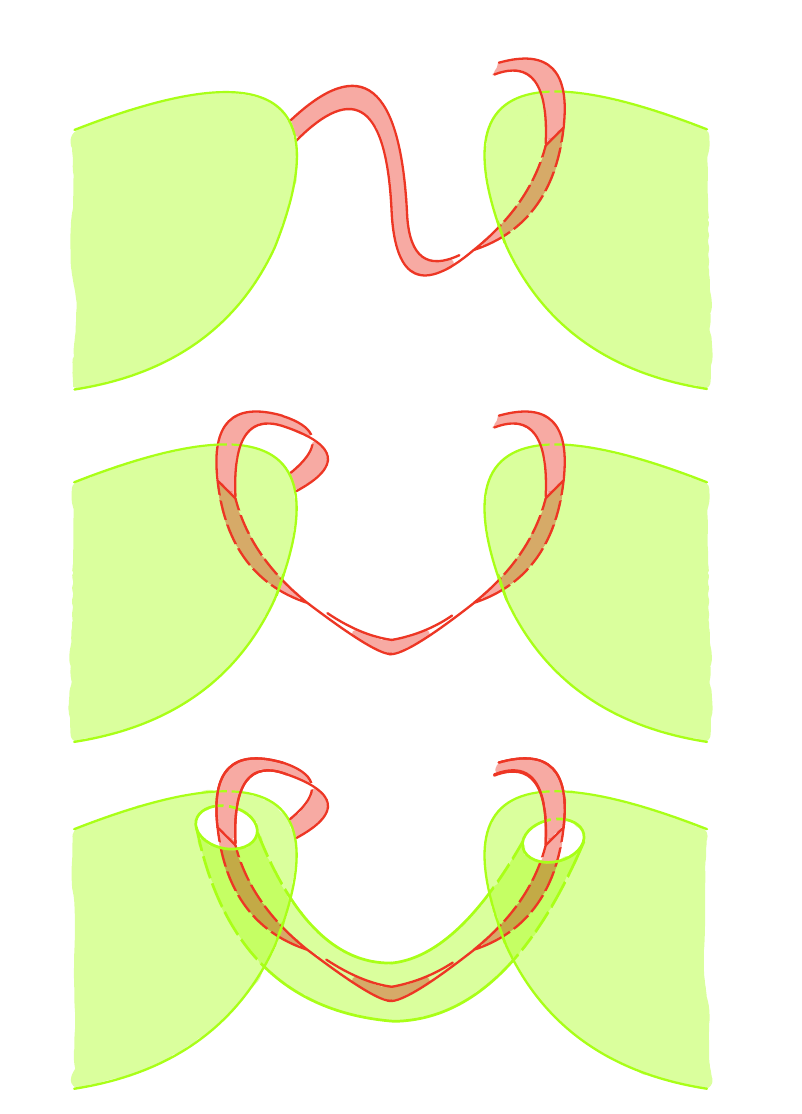}
        \put(7,50){\textcolor{green}{\Large{$S$}}}
        \put(7,81){\textcolor{green}{\Large{$S$}}}
        \put(7,19){\textcolor{green}{\Large{$S$}}}
        \put(30,23){\textcolor{red}{$b$}}
        \put(30,54){\textcolor{red}{$b$}}
        \put(30,85){\textcolor{red}{$b$}}
        \put(45,85){\textcolor{red}{$r_1$}}
        \put(45,54){\textcolor{red}{$r_1$}}
        \put(17,54){\textcolor{red}{$r$}}
    \end{overpic}
    \caption{Top: first intersection between $b$ and $S$. Middle: adding a ribbon intersection $r$ between $S$ and $b$. Bottom: removing $r$ and $r_1$ by adding a tube to $S$.}
    \label{fig:kink}
\end{figure}
We then carve out from $S$ two small disk neighbourhoods of $r$ and $r_1$ respectively, and join the resulting $S^1$-boundary components with a tube following the band $b$; refer again to Figure \ref{fig:kink}. 
Note that we can always add a kink as in Figure \ref{fig:kink} so that the resulting tube attachment yields again an orientable surface. In this way, we produce an oriented Seifert surface for $L$ having one less ribbon intersection with $b$. One can then inductively remove all the double points of $b \cap S$ and  get the desired Seifert surface, which we also denote by $S$. 

At this point, we have a Seifert surface $S$ which is disjoint from the band $b$ involved in the cobordism. Let $\beta \subset S^\circ$ be a mod $n$ characteristic link corresponding to the Fox coloring of the branch set $L$. Then $S'=S \cup b$ is a connected Seifert surface for $L'$ and Lemma \ref{lem:bandsum} tells us that a mod $n$ characteristic link for $L'$ is just given by the inclusion $\beta'$ of $\beta$ in $S'^\circ$ via $S^\circ \subset S'^\circ$. It is immediate to check that \[\kappa_n([\beta])=\kappa_n([\beta'])\] since, by the last sentence of Lemma~\ref{lem:bandsum}, the self-pairings of $\beta$ and $\beta'$ are equal on the nose.

We now address the case where $L'$ is obtained from $L$ by an embedded colored $2$-handle attachment. As usual, we let $S$ denote a Seifert surface for $L$. We first construct a new connected Seifert surface for $L$ which is disjoint from the given $2$-handle. Let $U$ be the colored unknotted split component of $L$ killed in the cobordism and let $D_U\subset S^3$ be the (Fox colored) $2$-disk capping it off. Take a connected Seifert surface $S''$ for $L'= L \setminus U$ disjoint from $D_U$ and let $S$ be the Seifert surface for $L$ obtained by joining $ S''$ and $D_U$ by an unknotted tube. Consider the Seifert surface for $L'$ defined by 
\[S'=S \cup D_U.\] 

It is easy to check that a mod $n$ characteristic link $\beta'$ for $L'$ is given by the inclusion of $\beta$ inside $S'$ via $S^\circ \subset S'^\circ$. We now claim that $\kappa_n([\beta])=\kappa_n([\beta'])$. Let $V'$ be a Seifert matrix for $S'$ with respect to a given basis of $H_1(S';\Z)$. Without loss of generality, we can assume that this basis is represented by curves disjoint from the disk $S'\backslash S$. Hence, we can complete this basis to one of $H_1(S;\Z)$ by adding the class of a parallel copy of the unknotted component $U$ in $S$; here we use that $U$ is a split component of $L$, so that $H_1(S;\Z)\cong H_1(S';\Z)\oplus \Z\langle [U]\rangle$. Since $U$ is split and unknotted, it has trivial linking number with every curve in $S'$, its self linking number is trivial and the associated Seifert matrix $V$ for $S$ is just $V' \oplus [ 0 ] $. Together with the fact that $\beta'$ is the image of $\beta$ under the inclusion $S \subset S'$, this concludes the argument.

We have shown that $\kappa_n$ induces a cobordism invariant
\[\kappa_n \colon Cob_{D_n}(S^3) \longrightarrow \Z_{2n}.\]
We now see that $\kappa_n$ is a group homomorphism. Indeed, let $\beta\subset S$ and $\beta'\subset S'$ be mod $n$ characteristic links for two $D_n$-covers, $p$ and $p'$, respectively. Using Lemma \ref{lem:transport}, one join $S$ and $S'$ via a tube to get a connected Seifert surface $S_{p+p'}$ for $p+p'$ and a mod $n$ characteristic link $\beta_{p+p'}\subset S_{p+p'}^\circ$ such that
\[[\beta_{p+p'}]^t\, A_{p+p'}\, [\beta_{p+p'}]=[\beta]^t \,A\, [\beta] + [\beta']^t \,A'\,[\beta'],\]
where $A, A'$ and $A_{p+q}$ denote symmetrized Seifert surfaces for $S, S'$ and $S_{p+q}$ respectively. Dividing by $n$ and reducing modulo $2n$, we obtain
\[\kappa_n([p]+[p'])=\kappa_n([p \sqcup p'])=\kappa_n([p])+\kappa_n([p']).\]
\end{proof}

In a similar fashion, we can upgrade the map $\kappa^{\leftrightarrow}_n$ in Lemma \ref{lem:nonorinv} to a non-orientable cobordism invariant.

\begin{proposition}\label{prop:cobordinvunor}
For every integer $n \geq 2$, the map $\kappa^{\leftrightarrow}_n$ from Lemma \ref{lem:nonorinv} induces a group homomorphism
\[\kappa^{\leftrightarrow}_n \colon Cob^{\leftrightarrow}_{D_n}(S^3) \longrightarrow \Z_n.\]

\end{proposition}
\begin{proof}
As already done in Proposition \ref{prop:cobordinv}, one can show that the assignment $\kappa^{\leftrightarrow}_n$ is invariant under elementary cobordisms. The proof is essentially the same, modulo the fact that non-orientable $1$-handles might be involved in a cobordism. We give a sketch of how to deal with that. Let $S$ be a connected Seifert surface for a Fox colored link $L$ and suppose that $L'$ is obtained from $L$ by an elementary cobordism corresponding to a single colored non-orientable $1$-handle attachment in $S^3 \times I$. In particular, we can view $L'$ as obtained from $L$ via a colored non-orientable band attachment in $S^3$. As in Proposition \ref{prop:cobordinvunor}, we can assume that $S$ is disjoint from the given band $b$, so that $S'=S \cup b$ is a connected spanning surface for $L'$. Lemma~\ref{lem:bandsum} implies that $\beta \subset S'$ is such that the Fox coloring of $L'$ restricts to
\[w \mapsto y^{lk(\beta,w)} \quad \text{for every } w \in \pi_1(E_{S'}),\]
and therefore, by Lemma~\ref{lem:glchar}, $U' \, [\beta]\equiv 0 \bmod{n}$, where $U'$ is a Gordon--Litherland matrix for $S'$. As in the proof of Lemma \ref{lem:nonorinv}, one can conclude that the congruence
\[[\beta]^t \, U' \, [\beta]\equiv [\beta']^t\, A' \, [\beta']  \pmod{n^2}\]
holds for any mod $n$ characteristic link $\beta'$ defining a representation which is inner equivalent to the one described by the Fox coloring of $L'$. Here $A'$ is a symmetrized Seifert matrix for the Seifert surface of $L'$ containing $\beta'$. If $p$ and $p'$ are the $D_n$-covers associated to the colored links $L$ and $L'$ and $A$ is a symmetrized Seifert matrix for $S$, it follows that
\[\kappa^{\leftrightarrow}_n([p])=\frac{1}{n} [\beta]^t \, A \, [\beta]=\frac{1}{n} [\beta]^t \, U' \, [\beta]=  \frac{1}{n} [\beta']^t \, A' \, [\beta'] = \kappa_n^{\leftrightarrow}([p']) \pmod{n}  ,\]
as desired.

The argument to show that $\kappa^{\leftrightarrow}_n$ is a group homomorphism is also identical to the one employed in Proposition \ref{prop:cobordinv} for $\kappa_n$ and is hence omitted.
\end{proof}

\color{black}
\section{Proof of Theorem \ref{thm:main}} \label{sec:proof}
\begin{proof}[Proof of Theorem \ref{thm:main}]

 When $n=1$, $D_n\cong\Z_2,$ so any two dihedral covers are cobordant and the groups are trivial. We hence let $n \geq 2$ in what follows. 
 
 By Proposition~\ref{prop:cobordinv}, $\kappa_n \colon Cob_{D_n}(S^3)\rightarrow \Z_{2n}$ is a well-defined group homomorphism. We easily check that $\kappa_n$ is injective.  Suppose $\kappa_n([p])=0 \in \Z_{2n}$. In the notation of Equation~\ref{eq:inv}, $\kappa_n([p])=a_1n+a_2 \bmod 2n$ with $0 \leq a_2 \leq n-1$. Thus, $\kappa_n([p])=0 \in \Z_{2n}$ implies that $a_2=0$ and $a_1 n \equiv 0 \bmod 2n$. It follows that $a_1$ is even. Proposition~\ref{prop:inj} then gives $[p]=0$.
 
Let us now study the image of $\kappa_n$. Since \[n\,\kappa_n([\beta])=[\beta]^t\, A\, [\beta]=2\,[\beta]^t\,V\,[\beta] \pmod{2n}\] is even, it follows that the integer $\kappa_n([\beta])$ is always even if $n$ is odd. This implies that, for $n$ odd, the image of $\kappa_n$ is contained in the subgroup of even residues of $\Z_{2n}$ (see also Remark~\ref{rem:knots}). The coloring of the $(2,n)$-torus link given in Example~\ref{ex:generator} satisfies 
 \[
 [\beta_n]^t\,A\,[\beta_n]= -n^2+n,
 \]
 which gives
 \[\kappa_n([\beta_n])=\frac{-n^2+n}{n}=1-n.\]
For $n$ even, $1-n$ generates $\Z_{2n}$ and $\kappa_n$ is onto. For $n$ odd, $1-n$ generates the subgroup of even residues, which is cyclic of order $n$. This implies that, in this case, the image of $\kappa_n$ is $\langle 2 \rangle \Z_{2n}\subset \Z_{2n}$. Together with injectivity, this proves the statements about $Cob_{D_n}(S^3)$.

Let us move to the non-orientable case. By Lemma~\ref{lem:nonorinv} and Proposition~\ref{prop:cobordinvunor}, $\kappa^{\leftrightarrow}_n ([p])=\kappa_n([p]) \bmod n$, where on the right-hand side we pick an arbitrary orientation on the branch set, is a well-defined group homomorphism  $Cob^{\leftrightarrow}_{D_n}(S^3)\to \Z_n$. If $\kappa^{\leftrightarrow}_n([p])=0$, then $a_2=0$ and Proposition~\ref{prop:injnon} gives $[p]=0 \in Cob^{\leftrightarrow}_{D_n}(S^3)$. The value of $\kappa^{\leftrightarrow}_n$ on the coloring of the $(2,n)$-torus link described in Example \ref{ex:generator} is $1-n\equiv 1 \bmod n$, so $\kappa^{\leftrightarrow}_n$ is onto, hence an isomorphism. \color{black}
\end{proof}

Using Theorem \ref{thm:main}, we can give a complete description of the following object.

\begin{definition}[Unoriented dihedral cobordism group]\label{def:unorientcob}
    Let $n \geq 1$ be an integer. The \textit{unoriented cobordism group of dihedral $n$-fold branched covers of $S^3$} is
\[Cob^{\mathfrak{u}}_{D_n}(S^3)
=
\left\{
p \colon Y\to S^3
\;\middle|\;
\substack{
p \text{ is a $D_n$-cover of $S^3$}\\
\text{branched over an unoriented link}
}
\right\}\big/\!\sim.
\]

where $p_0\sim p_1$ if and only if there exists a $D_n$-cover \[q \colon W \darrow{n} S^3 \times I\] branched over a smoothly embedded \textit{orientable} but \textit{unoriented} surface, such that the restriction of $q$ over $S^3\times \{0\}$ (resp. $S^3 \times \{1\}$) equals $p_0$ (resp. $p_1$). 

We endow $Cob^{\mathfrak{u}}_{D_n}(S^3)$ with an abelian group structure, in a manner identical to that in Definitions \ref{def:orientedcob} and \ref{def:unorcob}.
\end{definition} 

\begin{corollary}\label{cor:cob}
   For every integer $n \geq 1$, there is a group isomorphism
   \[Cob^{\mathfrak{u}}_{D_n}(S^3) \xlongrightarrow{\cong} Cob^{\leftrightarrow}_{D_n}(S^3),\]
   which is induced by the identity map on the set of $D_n$-covers.
\end{corollary}
In other words, two $D_n$-covers $p_0$ and $p_1$ branched over oriented links $L_0$ and $L_1$ cobound a non-orientable surface (over which the monodromies extend) if and only if they also cobound an orientable surface $F\subset S^3\times [0,1]$ (also extending the monodromies) such that the restriction of either orientation of $F$ to $\partial F$ does not necessarily respect the given orientations of $L_0$ and $L_1$.

\begin{proof}
    It is straightforward to observe that the identity map on the set of $D_n$-covers defines a surjective group homomorphism
\[Cob^{\mathfrak{u}}_{D_n}(S^3) \twoheadrightarrow Cob^{\leftrightarrow}_{D_n}(S^3).\]
In order to conclude, we have to show that this map is injective, i.e. that it has trivial kernel. Indeed, if $[p]=0\in Cob^{\leftrightarrow}_{D_n}(S^3)$, then Theorem \ref{thm:main} implies that $\kappa^{\leftrightarrow}_n([p])=0$. As a consequence, the monodromy $\omega_p$ can be described by means of a mod $n$ characteristic link $\beta \subset S^\circ$ satisfying $[\beta]^t \, A \, [\beta] \equiv 0 \bmod{n^2}$. Proposition \ref{prop:injnon} and Remark \ref{rmk:inj} then imply that $[p]=0\in Cob^{\mathfrak{u}}_{D_n}(S^3)$ as well. 
\end{proof}
\color{black}
\section{Modifying cobordism classes via moves à la Montesinos}\label{sec:mont}

In this section, we recall the following generalization of Montesinos moves for $3$-fold irregular dihedral covers to the setting of $D_n$-covers. 
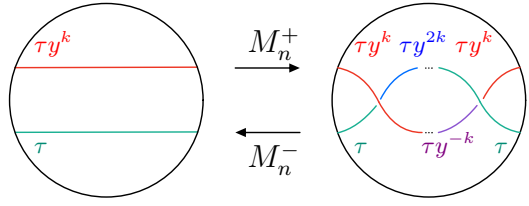
\begin{figure}
    \centering
     \begin{overpic}[width=0.45\linewidth]{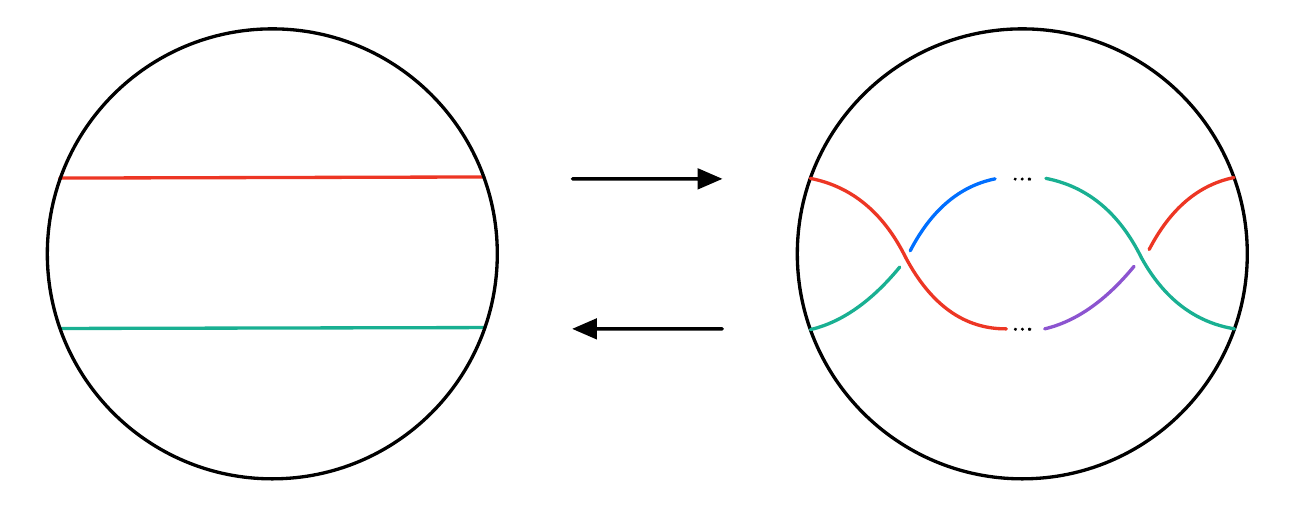}
        \put(46,29){\large{$M^+_n$}}
        \put(46,8.5){\large{$M^-_n$}}
        \put(8,29){\small{\textcolor{red}{$\tau y^k$}}}
        \put(8,11){\small{\textcolor{PineGreen}{$\tau$}}}
        \put(65,11){\small{\textcolor{PineGreen}{$\tau$}}}
        \put(90,11){\small{\textcolor{PineGreen}{$\tau$}}}
        \put(83,29){\small{\textcolor{red}{$\tau y^k$}}}
        \put(65,29){\small{\textcolor{red}{$\tau y^k$}}}
        \put(73,29){\small{\textcolor{blue}{$\tau y^{2k}$}}}
        \put(77,11){\small{\textcolor{violet}{$\tau y^{-k}$}}}        
    \end{overpic}
    \caption{The Montesinos moves $M^+_n$ and $M^-_n$. The twist region on the right contains $n$ crossings.}
    \label{fig:montesinos}
\end{figure}

\begin{definition}[Montesinos $n$-move]\label{def:mont}
Let $n \geq 2$ and let $p \colon Y \rightarrow S^3$ be a $D_n$-cover. We say that a $D_n$-cover $p'$ is obtained via a \textit{Montesinos $n$-move, $M^+_n$,} on $p$ if the following condition holds. There is an embedded $3$-disk $D \subset S^3$ intersecting the branch set of $p$ along two properly embedded unknotted arcs, labeled by reflections $\tau$ and $\tau y^k$ respectively, where $k \in \Z$ is such that $\gcd (n,k)=1$; and the Fox colored branch set of $p'$ is obtained from that of $p$ by adding $n$ half twists to the intersection of the branching set and the $3$-ball $D$; see Figure \ref{fig:montesinos}. We write $p'=p+M^+_n$. The inverse of a Montesinos $M^+_n$ move is denoted a Montesinos $M^-_n$ move. 
\end{definition}

We remark that, when $Y$ is connected, one can always find an embedded $3$-disk $D \subset S^3$ intersecting the branch set of $p$ as above so a Montesinos $n$-move can be performed. 
Next, we show that performing a Montesinos $n$-move as in Definition \ref{def:mont} does not change the total space of the given cover.

\begin{lemma}\label{lem:montinv}
 Let $n \geq 2$ be an integer and $p$ a $D_n$-cover with connected total space. The total space of the $D_n$-cover $p + M^\pm_n$ is homeomorphic to the one of $p$. 
\end{lemma}
   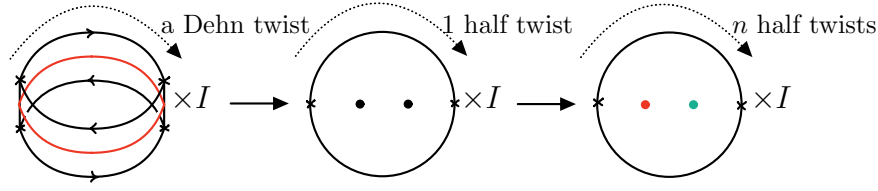
\begin{figure}
    \centering
     \begin{overpic}[width=0.75\linewidth]{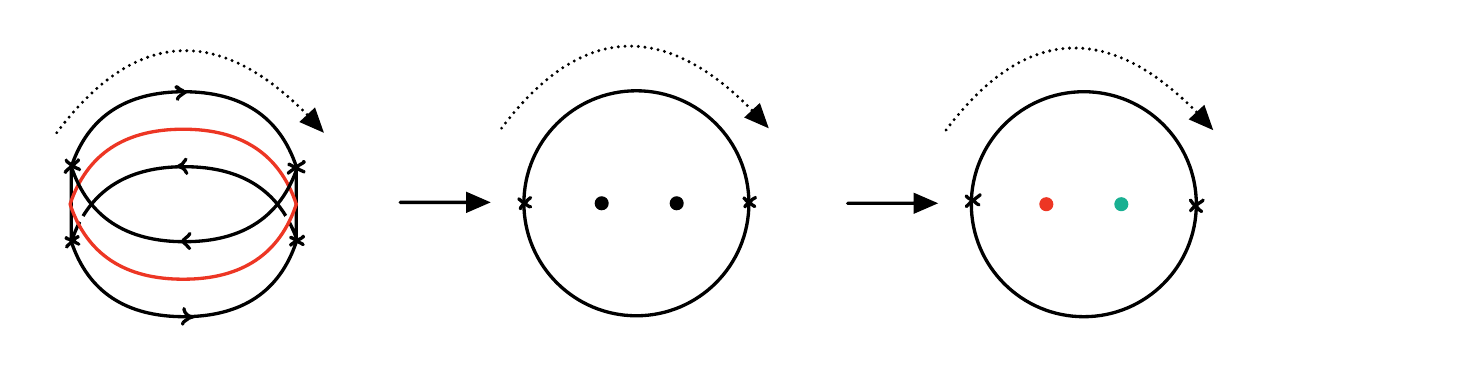}
        \put(21,12){\large{$\times I$}} 
        \put(52,12){\large{$\times I$}}
        \put(83,12){\large{$\times I$}}
        \put(81,20){$n$ half twists}
        \put(50,20){$1$ half twist}
        \put(20,20){a Dehn twist}
    \end{overpic}
    \caption{    Right: the $3$-disk $D\cong D^2\times I$ with a cross-section of the colored branch set of the $n$-fold irregular cover $\widetilde D \rightarrow D$. The branch set can be represented by two parallel strands. A Montesinos $n-$move is equivalent to performing $n$ half-twists on those strands. Middle: the $3$-disk $\widetilde D\cong D^2\times I$. The (pre-move) branch set of the cover $\widetilde D \rightarrow D$ lifts to two parallel strands; the Montesinos $n$-move lifts to a single half-twist. The two embedded arcs are also the branch set of the double cover $T \rightarrow \widetilde D$. Left: the solid torus $T$, written as the product of an annulus and an interval. The annulus is a double cover of $\widetilde D$ branched along the two black points. The half-twist on $\widetilde D$ lifts to a Dehn twist along the red curve, which extends over the solid torus.    }
    \label{fig:montesinos2}
\end{figure}
\begin{proof}
 The total space of $p$ is homeomorphic to a double branched cover of the total space of an associated irregular dihedral cover $p' \colon Y' \rightarrow S^3$, where $Y'$ is the quotient of $Y$ by the action of a fixed reflection (see the discussion after Definition \ref{def:Dncover}). Consider the $3$-ball $D$ involved in the Montesinos move (depicted in Figure \ref{fig:montesinos}) and write $D= D^2 \times I$, where $(D^2\times \{t\}) \cap B_p$ consists of two points for every $t \in I$ and the twisting performed during the Montesinos move happens along the $I$-factor. We can regard an $n$-move on the branching link as realized by carving out $D$ from $S^3$ and re-gluing it via a self-diffeomorphism on the boundary. Specifically, we can use a self-diffeomorphism of $\partial D$ which is the identity on $D^2 \times \{0\}$; a concatenation of $n$ (clockwise or counter-clockwise, depending on the sign) half twists of the $S^1$-factor on $S^1 \times I$; and the identity (resp. a rotation by $\pm \pi$ around the origin) on $D^2 \times \{1\}$ when $n$ is even (resp. $n$ is odd). 
 
 The lift of $D$ to the irregular cover $p'$ is again a $3$-ball (this is well known and easily seen by an Euler characteristic argument) which we denote by $\widetilde D$. We also write $\widetilde D = \widetilde D^2 \times I$, where $\widetilde D^2 \times \{ t\}$ covers $D^2\times \{t\}$ in degree $n$ for every $t \in I$. Similarly, $\partial \widetilde D^2 \times \{ t\}\to \partial D^2\times \{t\}$ is the $n$-to-1 cover $S^1\to S^1$. It follows that the lift via $p'$ of the self-diffeomorphism of $\partial D$ used to perform the Montesinos $n$-move is a single half twist of the $S^1$-factor on $\partial (\widetilde D^2 )\times I\cong S^1 \times I$, the identity on $\widetilde D^2 \times \{0\}$ and (regardless of the parity of $n$) a rotation by $\pm \pi$ around the origin on $\widetilde D^2 \times \{1\}$. 

 We now consider the lift of $\widetilde D$ in the regular cover. This is the double cover of $\widetilde D$, branched along two trivially embedded proper arcs, i.e. a solid torus $T \cong (S^1 \times I) \times [0,1]$, where $(S^1 \times I)\times \{t\}$ doubly covers $ \widetilde D^2 \times \{t\}$ for every $t \in [0,1]$. Figure~\ref{fig:montesinos2} depicts of $T$, $\widetilde D$ and $D$. It is a standard exercise to see that the given self-diffeomorphism of $\partial \widetilde D$ lifts to a map on $\partial (S^1 \times I \times [0, 1])$, which is the identity on $S^1 \times I \times \{0\}$ and a Dehn twist on the annulus $S^1 \times I \times \{1\}$ along its core. On $S^1 \times \partial I \times [0,1]$, the lift is an isotopy (with time parameter given by the last $I$-factor) between the identity map of $S^1 \times \partial I \times \{0\}$ and the result of performing two opposite half twists on the two connected components of $S^1 \times \partial I\times\{1\}$. In $S^1\times I\times \{1\}$, moving along the $I$ direction toward the core of the annulus, each of these half-twists peters out to zero, so that the core of $S^1\times I\times \{1\}$ is fixed. 
 
 The last step is to show that this self-map on the boundary extends to the interior of the solid torus $T$. Indeed, we may regard  $T=(S^1 \times I) \times [0,1]$ as the trace of the natural isotopy between the identity on $S^1\times I\times \{0\}$ and the result of performing on $S^1\times I\times \{1\}$ the self-map described in the previous paragraph. The circle $S^1\times \{\frac{1}{2}\}$ is fixed during this isotopy, and the trace of the restriction of the isotopy to $S^1\times \partial I$ is exactly the lift via $p$ of the given self-diffeomorphism of $\partial D$. 
\end{proof}
In the next lemma, we explain how Montesinos $n$-moves change the cobordism class of a given $D_n$-cover. 
\begin{lemma}\label{lem:modify}
Let $p$ be a $D_n$-cover with connected total space branched over an oriented link. Then $p \pm M^\pm_n$ is cobordant to the split sum of $p$ with the connected cover determined by a $(2,n)$-torus link, see Remark \ref{rmk:torus}.
\end{lemma}
\begin{proof}
 A cobordism is depicted in Figure \ref{fig:cob}. Alternatively, one may use the argument given in the proof of Corollary~\ref{cor:4-ball-covered}.
     \begin{figure}
    \centering
     \begin{overpic}[width=0.45\linewidth]{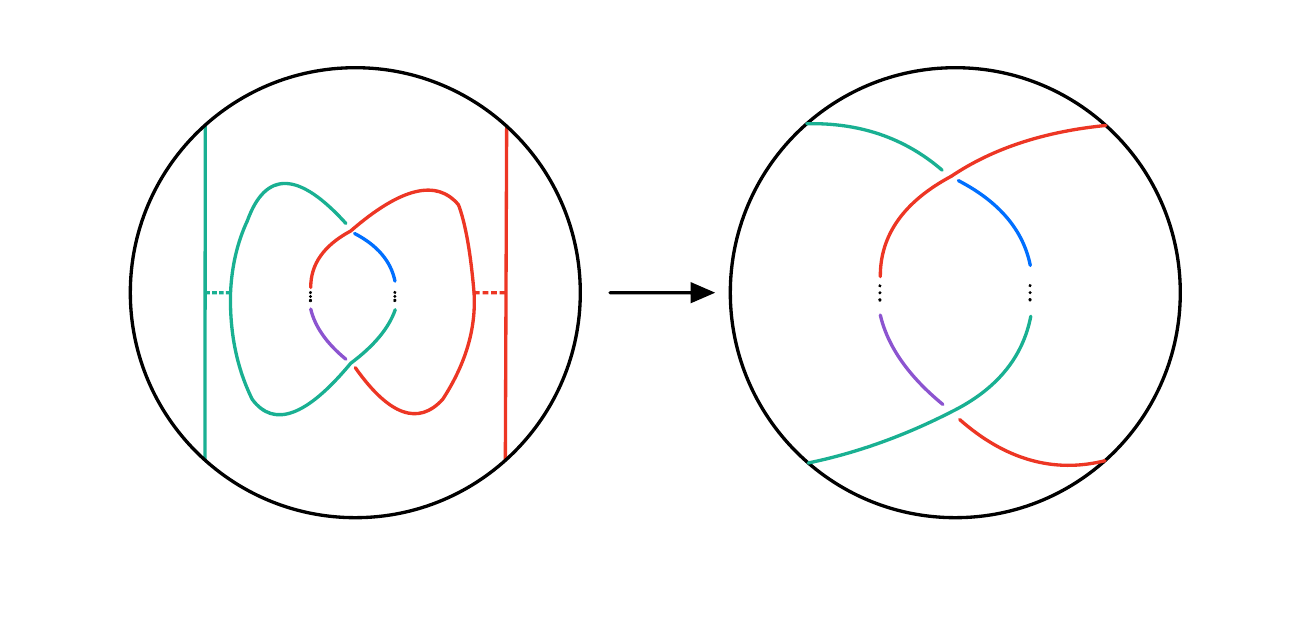}
         
    \end{overpic}
    \caption{A cobordism from $p \sqcup T(2,n)$ to $p+M^+_n$, obtained by attaching colored bands along the two dashed lines on the left side of the figure.}
    \label{fig:cob}
\end{figure}
\end{proof}
\begin{proof}[Proof of Theorem \ref{thm:C}]
  Let $p$ be a $D_n$-cover. Since the connected cover $T(2,n)$ branched over a $(2,n)$-torus link generates $Cob_{D_n}(S^3)$ (see Remark \ref{rmk:torus}), there exists an integer $m$ such that \[[p]+m\, [T(2,n)]=0\in Cob_{D_n}(S^3).\]  By Lemma \ref{lem:modify}, we have that $[p]+m \, [T(2,n)]=[p']$, where $p'=p+|m| \, M^\pm_n$. Lemma \ref{lem:montinv} then implies that $p$ and $p'$ have the same total space and, since $[p']=0$ in $Cob_{D_n}$, the conclusion follows.
\end{proof}
\color{black}
\begin{proof}[Proof of Corollary~\ref{cor:4-ball-covered}]
    Given $Y^3$ as above, there exists a $3$-fold irregular dihedral cover $f\colon Y\to S^3$ branched over some knot $K$ (see \cite{Hi}, \cite{Hirsch} and \cite{M2}). Let $\beta$ be a mod~$3$ characteristic knot for $K$, inducing $f$ in the sense of Equation \ref{eq:rho}, and contained in a Seifert surface $S$ for $K$. Regarding $S$ as a union of a disk and some bands, we know that $\beta$ passes through at least one non-trivially colored band $b$ in $S$. Moreover, after possibly passing to another characteristic knot in the same equivalence class, we may assume that $\beta$ passes through $b$ exactly once (with either positive or negative orientation). Performing two Montesinos 3-moves on $b$ induces a new 3-fold dihedral cover $f': Y\to S^3$ whose corresponding characteristic knot is the image of $\beta$ under the twisting operation. It is a straightforward computation that this operation changes the self-linking number of $\beta$ (with respect to the push-off determined by the Seifert surface) by $\pm 3$, depending on the sign of the $6$ added half-twists. We thus have $\kappa_3([f'])=\kappa_3([f])\pm 2$. Iterating this operation as needed, we arrive at a $3$-fold irregular dihedral cover $Y\to S^3$ which represents the trivial element of $Cob_{D_3}(S^3)$.
\end{proof}

\section*{Note on the use of AI}
In the course of preparing this manuscript, we used Claude and ChatGpt to verify proofs and improve the exposition. AI proposed stating a single result, Theorem A, which covers both the oriented and non-orientable case. In a previous draft, those were treated separately. 

\section*{Acknowledgements} This project started at the {\it Trisections and related topics} research school at CIRM in October of 2025. A first draft was completed during the conference {\it Singularities in topology and physics} at the University of Notre Dame in August 2026. We gratefully acknowledge funding from NSF DMS 2555770. VB is partially supported by GNSAGA – Istituto Nazionale di Alta Matematica ``Francesco Se\-ve\-ri'', Italy. During the course of this work, AK was partially supported by NSF DMS 2204349 and by a Simons Foundation Travel Grant for Mathematicians.

\section*{The Colored Fruit}

\textit{Each cover branched on a link in $S^3$}\\
\textit{Hanging down my dihedral tree}\\
\textit{Is cobordant to a chorus}\\
\textit{Of twin links on the torus}\\
\textit{A surface conducts the monodromy}

\bigskip
\medskip
\medskip

\noindent
\textsc{Valentina Bais}\\
Department of Mathematics, SISSA\\
Via Bonomea 265, 34136 Trieste, Italy\\
\texttt{vbais@sissa.it}

\medskip
\medskip
\medskip

\noindent
\textsc{Alexandra Kjuchukova}\\
Department of Mathematics, University of Notre Dame\\
Notre Dame, Indiana 46656, USA\\
\texttt{akjuchuk@nd.edu}
\end{document}